\documentclass{tac}  %@@
\usepackage[all,2cell]{xy}
\usepackage{mathrsfs,amssymb,stmaryrd,bbm,mathtools, mathabx}

\newtheorem{theo}{Theorem}[section]
\newtheorem{lem}[theo]{Lemma}
\newtheorem{prop}[theo]{Proposition}
\newtheorem{coro}[theo]{Corollary}
\newtheoremrm{defi}[theo]{Definition}
\newtheoremrm{rem}[theo]{Remark}

\def\mathrmdef#1{\expandafter\def\csname#1\endcsname{{\rm#1}}}
\def\mathsfdef#1{\expandafter\def\csname#1\endcsname{{\rm\sf#1}}}
\def\mathcaldef#1{\expandafter\def\csname#1\endcsname{{\mathcal#1}}}
\mathrmdef{id}\mathrmdef{op}
\mathsfdef{Alg}  \mathsfdef{Set} \mathsfdef{CAT}\mathcaldef{Lan}\mathrmdef{bi}
\mathrmdef{j}\mathrmdef{Id}\mathrmdef{t}\mathsfdef{Cat}\mathsfdef{cat}

\def\alg{\mathfrak{alg}}

\def\eeee{\mathfrak{e}}

\def\tttt{\mathfrak{t}}

\def\aaa{\mathfrak{A}}
\def\bbb{\mathfrak{B}}
\def\ccc{\mathfrak{C}}
\def\uuu{\mathcal{U}}

\def\sss{\mathfrak{S}}

\def\DDDD{\mathrm{D}}
\def\AAAA{\mathcal{A}}

\def\WWWW{\mathcal{W}}

\def\VVVV{\mathcal{V}}
\def\YYYY{\mathcal{Y}}

\def\TTTTT{\mathcal{T}}

\begin{document}  
\title{On lifting of biadjoints and lax algebras} 
\author{Fernando Lucatelli Nunes}
\address{CMUC, Department of Mathematics, University of Coimbra, 3001-501 Coimbra, Portugal}
\eaddress{lucatellinunes@student.uc.pt}


\thanks{This work was supported by CNPq, National Council for Scientific and Technological Development -- Brazil (245328/2012-2),  and by the Centre for Mathematics of the
University of Coimbra -- UID/MAT/00324/2013, funded by the Portuguese
Government through FCT/MCTES and co-funded by the European Regional Development Fund through the Partnership Agreement PT2020.}

\maketitle

\begin{abstract}
By the biadjoint triangle theorem, given a pseudomonad $\TTTTT $ on a $2$-category $\bbb $, if a right biadjoint $\aaa\to\bbb $ has a lifting to the pseudoalgebras $\aaa\to\mathsf{Ps}\textrm{-}\TTTTT\textrm{-}\Alg $ then this lifting is also right biadjoint provided that $\aaa $ has codescent objects. In this paper, 
we give general results on lifting of biadjoints. As a consequence, we get
 a \textit{biadjoint triangle theorem} which, in particular, allows us to study  triangles involving the $2$-category of lax algebras, proving analogues of the result described above. 
In the context of lax algebras, denoting by $\ell :\mathsf{Lax}\textrm{-}\TTTTT\textrm{-}\Alg \to\mathsf{Lax}\textrm{-}\TTTTT\textrm{-}\Alg _\ell $ the inclusion, if $R: \aaa\to\bbb $
is right biadjoint and has a lifting $J: \aaa\to \mathsf{Lax}\textrm{-}\TTTTT\textrm{-}\Alg $, 
then $\ell\circ J$ is right biadjoint as well provided that $\aaa $ has some needed weighted bicolimits. 
In order to prove such result, we study \textit{descent objects} and \textit{lax descent objects}.
At the last section, we study direct consequences of our theorems in the context of the 
\textit{$2$-monadic approach to coherence}. 
\end{abstract}

\section*{Introduction}
This paper has three main theorems. One of them (Theorem \ref{almostmain}) is about lifting of biadjoints: a generalization of Theorem 4.4 of \cite{Lucatelli1}. The others (Theorem \ref{GOAL} and Theorem \ref{GOAL2}) are consequences of the former on lifting biadjoints to the $2$-category of lax algebras. These results can be seen as part of what is called \textit{two-dimensional universal algebra}, or, more precisely, \textit{two-dimensional monad theory}: for an idea of the scope of this field (with applications), see for instance \cite{Power, Power89, Power2000, SLACKK, SLACK2020, HER, Bourke, Lucatelli1, Lucatelli2}.

There are several theorems about lifting of adjunctions in the literature~\cite{Barr, BORGERTHOLEN, Johnstone, STWW, Tholen}, including, for instance, adjoint triangle theorems~\cite{Dubuc, BarrWells}. Although some of these results can be proved for enriched categories or more general contexts~\cite{POWERR, Lucatelli1}, they often are not enough to deal with problems within $2$-dimensional category theory. The reason is that these problems involve concepts that are not of (strict/usual) $\Cat$-enriched category theory nature, as it is explained in \cite{CompLACK, Power}. 

For example, in $2$-dimensional category theory, the enriched notion of monad, the $2 $-monad, gives rise to the $2$-category of (strict/enriched) algebras, but it also gives rise to the $2$-category of pseudoalgebras and the $2$-category of lax algebras. The last two types of $2$-categories of algebras (and full sub-$2$-categories of them) are usually of the most interest despite the fact that they  are not ``strict'' notions.  

In short, most of the aspects of $2$-dimensional universal algebra are not covered by the usual $\Cat$-enriched category theory of \cite{Kelly, Dubuc2} or by the \textit{formal theory of monads} of \cite{RS72}. Actually, in the context of pseudomonad theory, the appropriate analogue of the formal theory of monads is the formal theory (and definition) of pseudomonads of \cite{MARMOLEJOD, SLACK2000}.
In this direction, the problem of lifting biadjunctions is the appropriate analogue of the problem of lifting adjunctions.

Some results on lifting of biadjunctions are consequences of the biadjoint triangle theorems proved in \cite{Lucatelli1}. One of these consequences is the following: let $\TTTTT $ be a pseudomonad on a $2$-category $\bbb $. Assume that  $R:\aaa\to\bbb $, $J: \aaa\to \mathsf{Ps}\textrm{-}\TTTTT\textrm{-}\Alg  $ are pseudofunctors such that we have the pseudonatural equivalence below. If $R$ is right biadjoint then $J$ is right biadjoint as well provided that $\aaa $ has some needed codescent objects.
$$\xymatrix{\aaa\ar[rr]^-{J}\ar[dr]_-{R}&&\mathsf{Ps}\textrm{-}\TTTTT\textrm{-}\Alg\ar[dl]^U\\
&\bbb\ar@{}[u]|-{\simeq } & }$$
One simple application of this result is, for instance, within the $2$-monadic approach to coherence~\cite{Lucatelli1}: roughly, the $2$-monadic approach to coherence is the study of biadjunctions and $2$-adjunctions between the many types of $2$-categories of algebras rising from a given $2$-monad. This allows us to prove ``general coherence results''~\cite{Power89, Power, SLACK2020} which encompass many coherence results -- such as the strict replacement of monoidal categories, the strict/flexible replacement of bicategories~\cite{SLACK2020, SLACK2007}, the strict/flexible replacement of pseudofunctors~\cite{Power} and so on~\cite{NICK}.  

If $\TTTTT ' $ is a $2$-monad, the result described above gives the construction of the left biadjoint to the inclusion
$$\TTTTT '\textrm{-}\Alg _ {\textrm{s}}\to\mathsf{Ps}\textrm{-}\TTTTT '\textrm{-}\Alg $$
subject to the existence of some codescent objects in $\TTTTT '\textrm{-}\Alg _ {\textrm{s}}$. The strict version of the biadjoint triangle theorem of \cite{Lucatelli1} shows when we can get a genuine left $2$-adjoint to this inclusion (and also studies when the unit is a pseudonatural equivalence), getting the coherence results of \cite{SLACK2020} w.r.t. pseudoalgebras.

In this paper, we prove Theorem \ref{almostmain} which is a generalization of Theorem 4.3 of \cite{Lucatelli1} on biadjoint triangles. Our result allows us to study lifting of biadjunctions to  lax algebras. Hence, we prove the analogue of the result described above for lax algebras. More precisely, let $\TTTTT $ be a pseudomonad on a $2$-category $\bbb $ and let $\ell: \mathsf{Lax}\textrm{-}\TTTTT\textrm{-}\Alg \to\mathsf{Lax}\textrm{-}\TTTTT\textrm{-}\Alg _ \ell$ be the locally full inclusion of the $2$-category of lax $\TTTTT $-algebras and $\TTTTT $-pseudomorphisms into the $2$-category of lax $\TTTTT $-algebras and lax $\TTTTT $-morphisms. Assuming that
$$\xymatrix{\aaa\ar[rr]^-{J}\ar[dr]_-{R}&&\mathsf{Lax}\textrm{-}\TTTTT\textrm{-}\Alg \ar[dl]\\
&\bbb\ar@{}[u]|-{\simeq } & }$$
is a pseudonatural equivalence in which $R$ is right biadjoint, we prove that $J$ is right biadjoint as well, provided that $\aaa $ has some needed codescent objects. Moreover, $\ell\circ J $ is right biadjoint if and only if $\aaa $ has lax codescent objects of some special diagrams. Still, we study when we can get  strict left $2$-adjoints to $J$ and $\ell\circ J$, provided that $J $ is a $2$-functor.

As an immediate application, we also prove general coherence theorems related to the work of \cite{SLACK2020}: we get the construction of the left biadjoints of the inclusions
$$\TTTTT '\textrm{-}\Alg _ {\textrm{s}}\to\mathsf{Lax}\textrm{-}\TTTTT '\textrm{-}\Alg _ \ell\qquad\qquad \mathsf{Ps}\textrm{-}\TTTTT  \textrm{-}\Alg \to\mathsf{Lax}\textrm{-}\TTTTT\textrm{-}\Alg _\ell$$
provided that $\TTTTT ' $ is a $2$-monad, $\TTTTT $ is a pseudomonad and $\TTTTT '\textrm{-}\Alg _ {\textrm{s}}$,  $\mathsf{Ps}\textrm{-}\TTTTT  \textrm{-}\Alg $ have some needed lax codescent objects. 

We start in Section \ref{Bilimits} establishing our setting: we recall basic results and definitions, such as weighted bicolimits and computads. In Section \ref{lift}, we give our main theorems on lifting of biadjoints: these are simple but pretty general results establishing basic techniques to prove theorems on lifting of biadjoints. These techniques apply to the context of \cite{Lucatelli1} but also apply to the study of other biadjoint triangles, such as our main application - which is the lifting of biadjoints to the $2$-category of lax algebras.

Then, we restrict our attention to $2$-dimensional monad theory: in order to do so, we present the weighted bicolimits called lax codescent objects and codescent objects in Section \ref{Descent}. Our approach to deal with descent objects is more general than the approach of \cite{RS87, Lucatelli1, Lucatelli2}, since it allows us to study descent objects of more general diagrams. 
Thanks to this approach, in Section \ref{Laxalgebras}, after defining pseudomonads and lax algebras, we show how we can get the category of pseudomorphisms between two lax algebras as a descent object at Proposition \ref{laxeffectivedescent}. This result also  shows how we can get the category of lax morphisms between two lax algebras as a lax descent object. 

In Section \ref{l}, we prove our main results on lax algebras:  Theorem \ref{GOAL} and Theorem \ref{GOAL2}. They are direct consequences of the results of Section \ref{lift} and Section \ref{Laxalgebras}, but we also give explicit calculations of the weighted bicolimits/weighted $2$-colimits needed in $\aaa $ to get the left biadjoints/left $2$-adjoints. We finish the paper in Section \ref{straightforward} giving straightforward applications of our results within the context of the $2$-monadic approach to coherence explained above.

This work was realized in the course of my PhD studies at University of Coimbra. I wish to thank my supervisor Maria Manuel Clementino for her support, attention and useful feedback.

\section{Preliminaries}\label{Bilimits}

In this section, we recall some basic results related to our setting, which is the tricategory $2\textrm{-}\CAT $ of $2$-categories, pseudofunctors, pseudonatural transformations and modifications. Most of what we need was originally presented in \cite{BE, RS76, RS80, RS87}. Also, for elements of enriched category theory, see \cite{Kelly}. We use the notation established in Section 2 of \cite{Lucatelli1} for pseudofunctors, pseudonatural transformations and modifications.

We start with considerations about size.
Let $\cat = \mathsf{int}\textrm{(}\Set\textrm{)} $ be the cartesian closed \textit{category of small categories}. Also, assume that $\Cat, \CAT $ are cartesian closed categories of categories in two different universes such that $\cat $ is an internal category of the subcategory of discrete categories of $\Cat $, while $\Cat$ is itself an internal category of the subcategory of discrete categories of $\CAT $.  Since these three categories of categories are complete and cartesian closed, they are enriched over themselves and they are cocomplete and complete in the enriched sense.

Henceforth, \textit{$\Cat$-category} is a $\Cat $-enriched category such that its collection of objects is a discrete category of $\CAT $. Thereby,  we have that $\Cat $-categories can be seen as internal categories of $\CAT $ such that their categories of objects are discrete. In other words, there is a full inclusion $\Cat\textrm{-}\CAT\to\mathsf{int}\textrm{(} \CAT\textrm{)}$ in which $\Cat\textrm{-}\CAT$ denotes the category of $\Cat$-categories. Moreover, since there is a forgetful functor $\mathsf{int}\textrm{(} \CAT\textrm{)}\to\CAT$, there is a forgetful functor $\Cat\textrm{-}\CAT\to\CAT $.

So, we adopt the following terminology: Firstly, a \textit{$2$-category} is a $\Cat $-category. Secondly, a \textit{possibly (locally) large $2$-category} is an internal category of $\CAT $ such that its category of objects is discrete. Finally, a \textit{small $2$-category} is a $2$-category which can be seen as an internal category of $\cat $. 

Let $\WWWW : \sss\to \Cat, \WWWW ' : \sss ^{\op }\to \Cat $ and $\DDDD: \sss\to\aaa $ be $2$-functors with small domains. If it exists, we denote the \textit{weighted limit} of $\DDDD $ with weight $\WWWW $ by  $\left\{ \WWWW, \DDDD\right\} $. Dually, we denote by $\WWWW '\Asterisk \DDDD $ the \textit{weighted colimit} provided that it exists. 

\begin{rem}
Consider the category, denoted in this remark by $\sss _ {\textrm{ist}}$ with two objects and two parallel arrows between them.
We can define the weight 
$$\WWWW _ {\textrm{insert}}: \sss _ {\textrm{ist}}\to\Cat $$
$$\xymatrix{ \mathsf{1} \ar@<-1.7 ex>[r]\ar@<1.7 ex>[r] & \mathsf{2}\ar@{|->}[rr]&& I \ar@<-1.7 ex>[r]_{\textrm{domain}}\ar@<1.7 ex>[r]^{\textrm{codomain} } & \textrm{2}
}$$
in which $\textrm{2}$ is the category with two objects and only one morphism between them and $I$ is the terminal category. The colimits with this weight are called \textit{coinserters} (see \cite{Kelly2}).
\end{rem}

The \textit{bicategorical Yoneda Lemma} says that  there is a pseudonatural equivalence $$[\sss , \Cat ]_ {PS} (\sss (a, - ), \mathcal{D})\simeq \DDDD a $$ given by the evaluation at the identity, in which $[\sss , \Cat ]_ {PS}$ is the possibly large $2$-category of pseudofunctors, pseudonatural transformations and modifications $\sss\to\Cat $. As a consequence, the Yoneda embedding $\YYYY _{{}_{\sss ^{\op } }} :\sss ^{\op }\to [\sss , \Cat ]_ {PS}$  is locally an equivalence (\textit{i.e.} it induces equivalences between the hom-categories).

If $ \WWWW  : \sss \to \Cat, \DDDD :\sss\to\aaa $ are pseudofunctors with a small domain, recall that the \textit{weighted bilimit}, when it exists, is an object $\left\{ \WWWW, \DDDD\right\} _ {\bi}$ of $\aaa $ endowed with a pseudonatural equivalence (in $X$)
$ \aaa (X, \left\{ \WWWW, \DDDD\right\} _ {\bi})\simeq  [\sss , \Cat ]_{PS}(\WWWW, \aaa (X, \DDDD -) )$.

The dual concept is that of weighted bicolimit: if $ \WWWW ' : \sss ^{\op } \to \Cat, \DDDD :\sss\to\aaa $ are pseudofunctors, the weighted bicolimit $ \WWWW '\Asterisk _ {\bi } \DDDD $ is the weighted bilimit $\left\{ \WWWW ', \DDDD ^{\op }\right\} _ {\bi}$ in $\aaa ^{\op } $. That is to say, it is an object $ \WWWW '\Asterisk _{\bi } \DDDD $ of $\aaa $ endowed with a pseudonatural equivalence (in $X$)
$ \aaa (\WWWW '\Asterisk _{\bi } \DDDD, X)\simeq  [\sss ^\op , \Cat ]_{PS}(\WWWW ', \aaa (\DDDD -, X ) )$.
By the bicategorical Yoneda Lemma,  $\left\{ \WWWW, \DDDD\right\}_ {\bi}, \WWWW '\Asterisk _{\bi } \DDDD $ are unique up to equivalence, if they exist.

\begin{rem}
If $\WWWW$ and $\DDDD $ are $2$-functors, $\left\{ \WWWW, \DDDD\right\} _ {\bi}$ and $\left\{ \WWWW, \DDDD\right\} $ may exist, without being equivalent to each other. This problem is related to the notion of flexible presheaves/weights (see \cite{Flexible}): whenever $\WWWW $ is \textit{flexible}, these two types of limits are equivalent, if they exist.
\end{rem}

\begin{defi}\label{ad}
Let 
$R:\aaa\to\bbb , E:\bbb\to\aaa $
be pseudofunctors. $E$ is
\textit{left biadjoint} to $R$ (or $R$ is \textit{right biadjoint} to $E$) if there exist
\begin{enumerate}
\item
pseudonatural transformations $\rho :\Id _ {\bbb } \longrightarrow RE $ and
$\varepsilon :ER \longrightarrow \Id _ { \aaa }$
\item
invertible modifications
$v : \id _{E} \Longrightarrow (\varepsilon E)  (E\rho )$
and
$w : (R\varepsilon)  (\rho R ) \Longrightarrow \id _{R}$
\end{enumerate}
satisfying coherence axioms~\cite{Lucatelli1}.
\normalsize
\end{defi}

\begin{rem}\label{addd}
Recall that a biadjunction $(E\dashv R, \rho , \varepsilon , v,w)$  has an associated pseudonatural equivalence 
$\chi : \bbb (-,R-)\simeq \aaa (E -, -)$,
in which  
\small
\begin{eqnarray*}
\chi _ {{}_{(X,Z)}}: &\bbb (X, R A) &\to \aaa (E X, A)\\
                &f        &\mapsto \varepsilon _ {{}_{A}} E(f)\\
								&\mathfrak{m} &\mapsto \id _ {{}_{\varepsilon _ {{}_{A}}}}\ast E(\mathfrak{m})
\end{eqnarray*}
\begin{eqnarray*}
\left(\chi _ {{}_{(h, g)}}\right) _ {{}_f} &:=& \left(\id _ {{}_{\varepsilon _ {{}_{A}} }}\ast \eeee _{{}_{(hf)R(g)}}\right) \cdot \left(\varepsilon _ {{}_{g}}\ast\eeee _ {{}_{hf}}\right).
\end{eqnarray*}		
\normalsize		
\end{rem}

If $L, U$ are $2$-functors, we say that $L$ is \textit{left $2$-adjoint} to $U$ whenever there is a biadjunction $(L\dashv  U, \eta , \varepsilon , s, t) $ in which $s, t $ are identities and $\eta , \varepsilon $ are $2$-natural transformations. In this case, we say that $(L\dashv U, \eta , \varepsilon ) $ is a \textit{$2$-adjunction}.

\subsection{On computads}
We employ the concept of computad, introduced in \cite{RS76}, to define the $2$-categories $\dot{\Delta } _ {\ell }, \dot{\Delta } , \Delta _ \ell $ in Section \ref{Descent}. For this reason, we give a short introduction to computads in this subsection.

Herein a \textit{graph} $\mathrm{G} = (d_1, d_ 0)$ is a pair of functors $d_0, d_ 1: \mathrm{G} _ {{}_{1}}\to \mathrm{G} _ {{}_{0}}$ between discrete categories of $\CAT $. In this case, $\mathrm{G} _ {{}_{0}}$ is called the collection of objects and, for each pair of objects $(a,b)$ of $\mathrm{G} _ {{}_{0}}$,  $d_ 0^{-1} (a)\cap d_1^{-1} (b) =\mathrm{G}(a,b) $ is the collection of arrows between $a$ and $b$.  A \textit{graph morphism} $\textrm{T}$ between $\mathrm{G}, \mathrm{G}'$ is a function $\textrm{T}:\mathrm{G} _ {{}_{0}}\to \mathrm{G} _ {{}_{0}}'$ endowed with a function $\textrm{T}_ {(a,b)}:\mathrm{G}(a,b)\to \mathrm{G}'(\textrm{T} a, \textrm{T} b)$ for each pair $(a,b)$ of objects in $\mathrm{G} _ {{}_{0}}$. That is to say, a graph morphism $\textrm{T}=(\textrm{T}_ {{}_{1}}, \textrm{T}_ {{}_{0}})$ is a natural transformation between graphs. The category of graphs is denoted by $\mathsf{GRPH}$.

We also define the full subcategories of $\mathsf{GRPH}$, denoted by $\mathsf{Grph}$ and $\mathsf{grph}$: the objects of $\mathsf{Grph}$ are graphs in the subcategory of discrete categories of $\Cat$ and the objects of $\mathsf{grph}$, called \textit{small graphs}, are graphs  in the subcategory of discrete categories of $\cat $. The forgetful functors $\CAT\to\mathsf{GRPH}$, $\Cat\to \mathsf{Grph} $ and $\cat\to \mathsf{grph} $ have left adjoints.

We denote by $\mathcal{F} : \mathsf{GRPH}\to  \CAT $ the functor left adjoint to $\CAT\to \mathsf{GRPH} $ and $ \mathfrak{F} $ the monad on $\mathsf{GRPH}$ induced by this adjunction. If $\mathrm{G}= (d_1, d_0) $ is an object of $\mathsf{GRPH}$, $\mathcal{F}\mathrm{G}$ is the coinserter of this diagram $(d_1, d_0)$.

Recall that $\mathcal{F}\mathrm{G}$, called the category freely generated by $\mathrm{G}$, can be seen as the category with the same objects of $\mathrm{G}$ but the arrows between two objects $a,b$ are the paths between $a,b$ (including the empty path): composition is defined by juxtaposition of paths.

\begin{defi}[Computad]
A \textit{computad} $\mathfrak{c}$ is a graph $\mathfrak{c}^ \mathrm{G}$ endowed with a graph $\mathfrak{c} (a,b)$  such that $\mathfrak{c} (a,b) _ {{}_{0}} = \left(\mathcal{F}\mathfrak{c}^ \mathrm{G}\right)(a,b) $ for each pair $(a,b)$ of objects of $\mathfrak{c}^\mathrm{G}$.
\end{defi}  

\begin{rem}\label{diaggraph}
A \textit{small computad} is a computad $\mathfrak{c}$ such that the graphs $\mathfrak{c}^ \mathrm{G}$ and $\mathfrak{c}(a,b)$ are small for every pair $(a,b) $ of objects of $\mathfrak{c}^ \mathrm{G}$. Such a computad can be entirely described by a diagram  
$$\xymatrix{  \mathfrak{c} _ {{}_{2}} \ar@<1ex>[r]^-{\partial _ 0 }\ar@<-1ex>[r]_-{\partial _1 } &\left( \mathfrak{F} \mathfrak{c}^ \mathrm{G}\right) _ {{}_{1}} 
\ar@<1ex>[r]^-{d _ 0 }\ar@<-1ex>[r]_-{d _1 } &\mathfrak{c}^ \mathrm{G} _ {{}_{0}}  }$$
in $\Set $ such that:
\begin{itemize}
\renewcommand\labelitemi{--}
\item  $(d_1,d_0) $ is the graph $\mathfrak{F} \mathfrak{c}^ \mathrm{G}$;
\item $\displaystyle\mathfrak{c} _ {{}_{2}}:= \bigcup _ {(a,b)\in\mathfrak{c}^ \mathrm{G} _ {{}_{0}}\times\mathfrak{c}^ \mathrm{G} _ {{}_{0}} } \mathfrak{c} (a,b) _ {{}_{1}}$;
\item $d_1\partial _ 1 = d_1\partial _ 0 $  and  $d_0\partial _ 1 = d_0\partial _ 0 $.
\end{itemize}
\end{rem}

A \textit{morphism $\textrm{T}$ between computads} $\mathfrak{c}, \mathfrak{c}'$ is a graph morphism  $\textrm{T}^\mathrm{G} : \mathfrak{c}^\mathrm{G}\to {\mathfrak{c}'}^ \mathrm{G} $ endowed with a graph morphism $\textrm{T}^ {(a,b)}: \mathfrak{c} (a,b)\to \mathfrak{c} (\textrm{T}^\mathrm{G} a,\textrm{T}^\mathrm{G} b)$  for each pair of objects $(a,b) $ in $\mathfrak{c}^ \mathrm{G}$ such that $\textrm{T}^ {(a,b)}_{{}_{0}}$ coincides with  $\mathcal{F}(\textrm{T}^\mathrm{G} )_{(a,b)}$. The \textit{category of computads} is denoted by $\mathsf{CMP} $.

We can define a forgetful functor $\mathcal{U}: \Cat\textrm{-}\CAT\to \mathsf{CMP} $ in which $\left(\mathcal{U}\aaa\right) ^ \mathrm{G}$ 
 is the underlying graph of the underlying category of $\aaa $. Recall that, for each pair of objects $(a,b)$ of $\left(\mathcal{U}\aaa\right) ^\mathrm{G}$,   an  object $f$ of $\left(\mathcal{U}\aaa\right)(a,b) $ is a path between $a$ and $b$. Then the composition defines a map   $\circ : \left(\mathcal{U}\aaa\right)(a,b)\to \aaa (a,b) $ and we can define the arrows of the graphs $\left(\mathcal{U}\aaa\right) (a,b)$ as follows:
$\left(\mathcal{U}\aaa\right) (a,b)(f,g) : = \aaa (a,b) (\circ (f), \circ(g))
$.

The left reflection of a small computad $\mathfrak{c} $ along $\mathcal{U}$ is denoted by $\mathcal{L}\mathfrak{c} $ and called the \textit{$2$-category freely generated by $\mathfrak{c} $}. The underlying category of $\mathcal{L}\mathfrak{c} $ is $\mathcal{F} \mathfrak{c}^ \mathrm{G}$ and its $2$-dimensional structure is constructed below.

$$\xymatrix{  \mathfrak{c} _ {{}_{2}}\coprod \left( \mathfrak{F} \mathfrak{c}^ \mathrm{G}\right) _ {{}_{1}}\ar@<0.5ex>[d]^-{\partial _ 0, \id }\ar@<-0.5ex>[d]_-{\partial _ 1, \id } \ar@<1ex>[rr]^-{d_0\partial _ 0, \mbox{ } d_0}\ar@<-1ex>[rr]_-{d_1\partial _ 0,\mbox{ } d_1} && \mathfrak{c}^ \mathrm{G} _ {{}_{0}}\ar[d]^-{\id}\\
\left( \mathfrak{F} \mathfrak{c}^ \mathrm{G}\right) _ {{}_{1}} 
\ar@<1ex>[rr]^-{d _ 0 }\ar@<-1ex>[rr]_-{d _1 } &&\mathfrak{c}^ \mathrm{G} _ {{}_{0}}  }$$
The diagram of Remark \ref{diaggraph} induces the graph morphisms $((\partial _ 0, \id ), \id) $ and $((\partial _ 1, \id), \id)$ above between a graph denoted by $\mathfrak{c}^- $ and $\mathfrak{F}\mathfrak{c} ^\mathrm{G}$. Using the multiplication of the monad $\mathfrak{F}$, these morphisms induce two morphisms $\mathfrak{F}\mathfrak{c}^-\to \mathfrak{F}\mathfrak{c} ^\mathrm{G}$. These two morphisms define in particular the graph $\mathfrak{c}_-$ below  and $\mathcal{F}\mathfrak{c}_-$ defines the $2$-dimensional structure of $\mathcal{L}\mathfrak{c} $.
$$\xymatrix{  \left(\mathfrak{F}\mathfrak{c}^-\right)_ {{}_{1}} -\left( \mathfrak{F}^2 \mathfrak{c}^ \mathrm{G}\right) _ {{}_{1}}  \ar@<1ex>[r]\ar@<-1ex>[r] & \left( \mathfrak{F} \mathfrak{c}^ \mathrm{G}\right) _ {{}_{1}}  }$$
Defining all compositions by juxtaposition, we have a sesquicategory (see \cite{categorical}). 
We define $\mathcal{L} \mathfrak{c} $ to be the $2$-category obtained from the quotient of this sesquicategory, forcing the interchange laws.

\begin{rem}
Let $\mathsf{Preord} $ the category of preordered sets. We have an inclusion $\mathsf{Preord}\to \Cat $ which is right adjoint. This adjunction induces a $2$-adjunction between $\mathsf{Preord}\textrm{-}\CAT $ and $\Cat\textrm{-}\textrm{CAT} $.

If $\mathfrak{c} $ is a computad, \textit{the locally preordered $2$-category freely generated by $\mathfrak{c} $} is the image of $\mathcal{L}\mathfrak{c} $ by the left $2$-adjoint functor $\Cat\textrm{-}\textrm{CAT}\to\mathsf{Preord}\textrm{-}\CAT $.

\end{rem}

\section{Lifting of biadjoints}\label{lift}

In this section, we assume that a small weight $\WWWW : \sss\to\Cat $, a right biadjoint pseudofunctor $R:\aaa\to\bbb $  and a pseudofunctor $J: \aaa\to\ccc $ are given. We investigate whether $J$ is right biadjoint. 

We establish Theorem \ref{almostmain} and its immediate corollary on biadjoint triangles. We omit the proof of Lemma  \ref{lemmainstrict}, since it is analogous to the proof of Lemma \ref{lemmain}.

\begin{lem}\label{lemmain}
Assume that, for each object $\mathtt{y}$ of $\ccc$, there are pseudofunctors $\mathfrak{D} _ {\mathtt{y} } :\sss\times \aaa \to  \Cat, \AAAA _ \mathtt{y}: \sss ^{\op }\to\aaa  $   such that $\mathfrak{D} _ {\mathtt{y}}\simeq \aaa(\AAAA _ \mathtt{y} -, -) $ and
$\left\{ \WWWW , \mathfrak{D}  _{\mathtt{y} }(-,A)\right\} _ {\bi }\simeq \ccc (\mathtt{y}, JA) $ for each object $A$ of $\aaa $. The pseudofunctor $J$ is right biadjoint if and only if, for every object $\mathtt{y}$ of $\ccc $, the weighted bicolimit  $\WWWW\Asterisk _ {\bi } \AAAA _ \mathtt{y} $ exists in $\aaa $. In this case, $J$ is right biadjoint to $G$, defined by $G\mathtt{y} = \WWWW\Asterisk _ {\bi } \AAAA _\mathtt{y} $.
\end{lem}
\begin{proof}
There is a pseudonatural equivalence (in $A$)   
$$\left\{\WWWW , \aaa(\AAAA _ \mathtt{y} -, A)\right\} _ {\bi} \simeq \left\{ \WWWW , \mathfrak{D}  _{\mathtt{y} }(-,A)\right\} _ {\bi } \simeq \ccc (\mathtt{y}, JA) .$$
Thereby, an object $G\mathtt{y}$ of $\aaa $ is the weighted bicolimit $\WWWW\Asterisk _ {\bi } \AAAA _\mathtt{y} $ if and only if there is a pseudonatural equivalence (in $A$) 
$\aaa (G\mathtt{y}, A)\simeq \left\{\WWWW , \aaa(\AAAA _ \mathtt{y} -, A)\right\} _ {\bi}\simeq \ccc (\mathtt{y}, JA)$.
That is to say, an object $G\mathtt{y} $ of $\aaa $ is the weighted bicolimit $\WWWW\Asterisk _ {\bi } \AAAA _\mathtt{y} $ if and only if $G\mathtt{y} $ is a birepresentation of $\ccc (\mathtt{y}, J-) $. 
\end{proof}

\begin{lem}\label{lemmainstrict}
Assume that $J , \WWWW  $  are $2$-functors and, for each object $\mathtt{y}$ of $\ccc$, there are $2$-functors $\mathfrak{D} _ {\mathtt{y} } :\sss\times \aaa \to  \Cat ,\AAAA _ \mathtt{y}: \sss ^{\op }\to\aaa  $   such that there is a $2$-natural isomorphism $\mathfrak{D} _ {\mathtt{y} }\cong \aaa(\AAAA _ \mathtt{y} -, -) $ and $\left\{ \WWWW , \mathfrak{D}  _{\mathtt{y} }(-,A)\right\} \cong \ccc (\mathtt{y}, JA) $ for every object $A$ of $\aaa $.
The $2$-functor $J$ is right $2$-adjoint if and only if, for every object $\mathtt{y}$ of $\ccc $, the weighted colimit  $\WWWW\Asterisk  \AAAA _ \mathtt{y} $ exists in $\aaa $. In this case, $J$ is right $2$-adjoint to $G$, defined by $G\mathtt{y} = \WWWW\Asterisk  \AAAA _\mathtt{y} $.
\end{lem}

Let $ \DDDD : \sss\times\aaa\to \Cat$ be a pseudofunctor. We denote by $\left|\DDDD \right| : \sss _ {{}_{0}}\times \aaa\to \Cat $ the restriction of $\DDDD $ in which $\sss _ {{}_{0}}$ is the discrete $2$-category  of the objects of $\sss $. Also, herein we say that  $\left|\DDDD \right| $ can be factorized through $R ^\ast := \bbb (-, R-)$ if there are a pseudofunctor $\DDDD ': \sss _ {{}_{0}}\to \bbb ^{\op }$ and a pseudonatural equivalence $ \left|\DDDD \right| \simeq  R^\ast\circ (\DDDD '\times \Id _ {{}_{\aaa }})$.

\begin{theo}\label{almostmain}
Assume that, for each object $\mathtt{y} $ of $\ccc$, there is a pseudofunctor $\mathfrak{D} _ {\mathtt{y}} :\sss \times \aaa \to \Cat $  such that $\left| \mathfrak{D} _ {\mathtt{y} }\right|$ can be factorized through $R^\ast $ and
$\left\{ \WWWW , \mathfrak{D}  _{\mathtt{y} }(-, A) \right\} _ {\bi }\simeq \ccc (\mathtt{y}, JA) $ for every object $A$ of $\aaa $. In this setting, for each object $\mathtt{y} $ of $\ccc $  there are a pseudofunctor $\AAAA _ \mathtt{y} :\sss ^{\op }\to\aaa  $ and a pseudonatural equivalence $\mathfrak{D} _ {\mathtt{y}}\simeq \aaa(\AAAA _ \mathtt{y} -, -) $.   

As a consequence, the pseudofunctor $J$ is right biadjoint if and only if, for every object $\mathtt{y}$ of $\ccc $, the weighted bicolimit  $\WWWW\Asterisk _ {\bi } \AAAA _ \mathtt{y} $ exists in $\aaa $. In this case, $J$ is right biadjoint to $G$, defined by $G\mathtt{y} = \WWWW\Asterisk _ {\bi } \AAAA _\mathtt{y} $.
\end{theo}

\begin{proof}
Indeed, if $E: \bbb\to\aaa $ is left biadjoint to $R$, then there is a pseudonatural equivalence $R^\ast\simeq \aaa (E-,-)$. Therefore, by the hypotheses, for each object $Y$ of $\ccc $, there is a pseudofunctor $\mathfrak{D} _ {\mathtt{y}} ': \sss _{{}_{0}} \to \bbb ^{\op }$ such that  
$\left| \mathfrak{D} _ {\mathtt{y}}\right|\simeq R^\ast\circ (\mathfrak{D}  ' _ {\mathtt{y}}\times \Id _ {{}_{\aaa }}) \simeq \aaa(E\mathfrak{D}'_\mathtt{y}-,-)$.
From the bicategorical Yoneda lemma, it follows that we can choose a pseudofunctor $\AAAA _ \mathtt{y}: \sss ^{\op }\to\aaa $ which is an extension of $E\mathfrak{D}'_\mathtt{y}$  such that 
$\aaa (\AAAA _ \mathtt{y} - , -)\simeq 	\mathfrak{D}_\mathtt{y} $. The consequence follows from Lemma \ref{lemmain}.
\end{proof}

\begin{coro}[Biadjoint Triangle]\label{used}
Assume that $\VVVV : \ccc ' \to\ccc  $ is a pseudofunctor and
$$\xymatrix{  \aaa\ar[rr]^-{J}\ar[dr]_-{R}&&\ccc '\ar[dl]^-{U}\\
&\bbb  & }$$
is a commutative triangle of pseudofunctors satisfying the following: for each object $\mathtt{y}$ of $\ccc$, there is a pseudofunctor  $\DDDD_ \mathtt{y} :\sss \times \ccc ' \to \Cat  $  such that
$\left|\DDDD _ \mathtt{y}\right|$ can be factorized through $U^\ast $ and $\left\{ \WWWW , \DDDD _\mathtt{y} (-, \mathtt{x}) \right\} _ {\bi }\simeq \ccc (\mathtt{y}, \VVVV \mathtt{x}) $ for each object $\mathtt{x}$ of $\ccc '$. In this setting, for each object $\mathtt{y}$ of $\ccc $,  there is a pseudofunctor $\AAAA _ \mathtt{y} :\sss ^{\op }\to\aaa  $ such that $\DDDD _ \mathtt{y}(-, J-)\simeq \aaa(\AAAA _ \mathtt{y} -, -) $.

As a consequence, the pseudofunctor $\VVVV \circ J$ is right biadjoint if and only if, for every object $\mathtt{y}$ of $\ccc $, the weighted bicolimit  $\WWWW\Asterisk _ {\bi } \AAAA _\mathtt{y}$ exists in $\aaa $. In this case, $\VVVV \circ J$ is right biadjoint to $G$, defined by $G\mathtt{y} = \WWWW\Asterisk _ {\bi } \AAAA _\mathtt{y} $. 
\end{coro}
\begin{proof}
We prove that $\mathfrak{D} _ {\mathtt{y}} := \DDDD _ \mathtt{y} (-,J-)$ satisfies the hypotheses of Theorem \ref{almostmain}. 
We have that, for each object $\mathtt{y}$ of $\ccc $ and each object $A$ of $\aaa $, 
 $$\left\{ \WWWW , \mathfrak{D}_\mathtt{y}(-,A)\right\} _ {\bi }\simeq \ccc (\mathtt{y}, \VVVV JA).$$ 

Also, for each object $\mathtt{y}$ of $\ccc $, there is a pseudofunctor $\DDDD _ \mathtt{y}': \sss _ {{}_{0}}\to\bbb ^{\op }$ such that  
$ U^\ast \circ (\DDDD _ \mathtt{y}'\times \Id _ {\ccc })\simeq \left| \DDDD _\mathtt{y}\right| $.
Therefore 
$$R^\ast \circ  (\DDDD _ \mathtt{y}'\times \Id _ {\aaa })\simeq U^\ast \circ (\DDDD _ \mathtt{y}'\times J )\simeq \left| \DDDD _\mathtt{y}\right|\circ (\Id _ {\sss _{{}_{0} }}\times J )\simeq \left| \mathfrak{D}_\mathtt{y}\right| .$$
\end{proof}

Corollary 5.10 of \cite{Lucatelli1} is a direct consequence of the last corollary and Proposition 5.7 of \cite{Lucatelli1}. In particular, if $\TTTTT $ is a pseudomonad on $\bbb $ and $U:\mathsf{Ps}\textrm{-}\TTTTT\textrm{-}\Alg\to\bbb $ is the forgetful $2$-functor, Proposition 5.5 of \cite{Lucatelli1} shows that the category of pseudomorphisms between two pseudoalgebras is given by a descent object (which is a type of weighted bilimit) of a diagram satisfying the hypotheses of Corollary \ref{used}. Therefore, assuming the existence of codescent objects in $\aaa $, $J$ has a left biadjoint.

In Section \ref{Laxalgebras}, we define the $2$-category of lax algebras of a pseudomonad $\TTTTT $. There, we also show Proposition \ref{laxeffectivedescent} which is precisely the analogue and a generalization of Proposition 5.5 of \cite{Lucatelli1}: the category of lax morphisms and the category of pseudomorphisms between lax algebras are given by appropriate types of weighted bilimits. Then, we can apply Corollary \ref{used} to get our desired result on lifting of biadjoints to the $2$-category of lax algebras: Theorem \ref{GOAL}. Next section, we define and study the weighted bilimits  appropriate to our problem, called lax descent objects and descent objects.

To finish this section, we get a trivial consequence of Corollary \ref{used}:

\begin{coro}\label{rt}
If $RJ=U$ are pseudofunctors in which $R$ is right biadjoint and $U$ is locally an equivalence, then $J$ is right biadjoint as well. Actually, if $E$ is left biadjoint to $R$, $G\mathtt{y} : = EU\mathtt{y} $ defines the pseudofunctor left biadjoint to $J$.
\end{coro}

\section{Lax descent objects}\label{Descent}
In this section we describe the $2$-categorical limits called lax descent objects and descent objects~\cite{RS76, RS80, RS87, RS, JanelidzeTholen, SLACK2020,  Lucatelli2}.

In page 177 of \cite{RS76}, without establishing the name ``lax descent objects'', it is shown that 
given a $2$-monad $\TTTTT $, for each pair $\mathtt{y}, \mathtt{z} $ of strict $\TTTTT $-algebras, there is a diagram
of categories for which its lax descent category (object) is the category of lax morphisms between $\mathtt{y} $ and $\mathtt{z} $.
We establish a generalization of this result for lax algebras: Proposition \ref{laxeffectivedescent}. 

In order to establish such result, our approach in defining the lax descent objects is different from \cite{RS76}, commencing with the definition of our ``domain $2$-category'', denoted by $\Delta _ {\ell } $.   
 
\begin{defi}[$\t : \Delta _\ell\to \dot{\Delta } _ \ell $ and $\j : \Delta _\ell\to\dot{\Delta }$]\label{delta}
We denote by $\dot{\mathfrak{\Delta } }_\ell $ the computad defined by the diagram
$$\xymatrix{  \mathsf{0} \ar[rr]^-d && \mathsf{1}\ar@<1.7 ex>[rrr]^-{d^0 }\ar@<-1.7ex>[rrr]_-{d^1 } &&& \mathsf{2}\ar[lll]|-{s^0}
\ar@<1.7 ex>[rrr]^{\partial ^0 }\ar[rrr]|-{\partial ^1}\ar@<-1.7ex>[rrr]_{\partial ^2} &&& \mathsf{3} }$$ 
with the $2$-cells:
\begin{eqnarray*}
\begin{aligned}
\sigma_{00} &:&  \partial^{0}d^{0}\Rightarrow \partial ^1 d ^0, \\
\sigma_{20} &:&  \partial ^2 d ^0\Rightarrow\partial^{0}d^{1}, \\
\sigma_{21} &:&  \partial ^2 d ^1\Rightarrow\partial^{1}d^{1},
\end{aligned}
\qquad
\begin{aligned}
 n_0        &:&  \id _{{}_\mathsf{1}}\Rightarrow s^0d^0 ,  \\
 n_1        &:&   \id _{{}_\mathsf{1}}\Rightarrow s^{0}d^{1}, \\
 \vartheta        &:&  d^1d\Rightarrow d^0d .
\end{aligned}
\end{eqnarray*}

The \textit{$2$-category $\dot{\Delta } _ \ell $} is, herein, the locally preordered $2$-category freely generated by $\dot{\mathfrak{\Delta } }_\ell $. The full sub-$2$-category of $\dot{\Delta } _ \ell $ with objects $\mathsf{1}, \mathsf{2}, \mathsf{3}$ is \textit{denoted by $\Delta _ \ell $} and the full inclusion by  $\t: \Delta _ \ell\to\dot{\Delta }_\ell $.

We consider also the computad $\dot{\mathfrak{\Delta } }$ which is defined as the computad $\dot{\mathfrak{\Delta } }_\ell $ with one extra $2$-cell $d^0d \Rightarrow d^1d $. We denote by $\dot{\Delta }$ the locally preordered $2$-category freely generated by $\dot{\mathfrak{\Delta } }$. Of course, there is also a full inclusion $\j : \Delta _\ell\to\dot{\Delta }$.
\end{defi}

We define, also, the computad $\mathfrak{\Delta } _ \ell $ which is the full subcomputad of $\dot{\mathfrak{\Delta } }_\ell $ with objects $\mathsf{1}, \mathsf{2}, \mathsf{3}$. 

\begin{prop}\label{Presentation}
Let $\aaa $ be a $2$-category. There is a bijection between the $2$-functors $ \Delta _ \ell\to \aaa $ and the maps of computads $\mathfrak{\Delta } _ \ell\to \uuu\aaa $. In other words, $\Delta _ \ell $ is the $2$-category freely generated by the computad $\mathfrak{\Delta } _ \ell $.

Also, there is a bijection between $2$-functors $\overline{\DDDD } : \dot{\Delta } _ \ell\to \aaa $ and the maps of computads $\DDDD   : \mathfrak{\Delta } _ \ell\to \uuu\aaa $ which satisfy the following equations:
\begin{itemize}
\renewcommand\labelitemi{--}
\item Associativity:
$$\xymatrix{  
\DDDD\mathsf{0}\ar[r]^{\DDDD (d)}\ar[dd]_{\DDDD (d)}\ar@{}[rdd]|-{\xRightarrow {\DDDD (\vartheta )} } 
&
\DDDD \mathsf{1}\ar[dd]^{\DDDD (d^0) }\ar@{}[rrddd]|{=}
&
&
\DDDD \mathsf{3}
&
&
\DDDD \mathsf{2}\ar[ll]^{\DDDD(\partial ^1)}\ar@{}[ld]|-{\xRightarrow{\DDDD(\sigma_{00} )}}
\\
&
&
&
&
\DDDD \mathsf{2}\ar@{}[l]|-{\xRightarrow{\DDDD(\sigma_{20} )} }\ar[lu]|-{\DDDD(\partial ^0)}
&
\DDDD \mathsf{1}\ar[l]_{\DDDD (d^0)}\ar[u]_-{\DDDD (d^0)}
\\
\DDDD \mathsf{1} \ar[r]^{\DDDD (d^1) }\ar[d]_{\DDDD (d^1)}\ar@{}[rd]|-{\xRightarrow {\DDDD(\sigma_{21} ) } }
&
\DDDD \mathsf{2}\ar[d]^{\DDDD(\partial ^1)} 
&
&
\DDDD \mathsf{2}\ar[uu]^{\DDDD(\partial ^2 )}\ar@{}[rd]|-{\xRightarrow {\DDDD (\vartheta )}}
&
\DDDD \mathsf{1}\ar[l]_ {\DDDD (d^0 ) }\ar@{}[r]|-{\xRightarrow{\DDDD (\vartheta )} }\ar[u]_{\DDDD (d^1)}
&
\\
\DDDD \mathsf{2}\ar[r]_{\DDDD(\partial ^2 )}
&
\DDDD \mathsf{3}
&
&
\DDDD\mathsf{1}\ar[u]^-{\DDDD (d^1)}
&
\DDDD\mathsf{0}\ar[l]^-{\DDDD (d)}\ar[u]_-{\DDDD (d)}\ar@/_6ex/[ruu]|-{\DDDD (d)}
&
}$$
\item Identity:
$$\xymatrix{
\DDDD\mathsf{0}\ar[r]^{\DDDD(d) }\ar[d]_{\DDDD(d)}\ar@{}[dr]|-{\xRightarrow{\DDDD(\vartheta ) } }
&
\DDDD\mathsf{1}\ar[d]^{\DDDD(d^0)}\ar@{}[ddrr]|-{=}
&
&
\DDDD\mathsf{0}\ar[d]_{\DDDD(d) } 
&
\\
\DDDD\mathsf{1}\ar[r]_{\DDDD(d^1 ) }\ar@{ { } }@/_1ex/[rd]|-{\xRightarrow{\DDDD(n_1)} }\ar@{=}@/_5ex/[rd]
&
\DDDD \mathsf{2} \ar[d]^{\DDDD (s^0)}
&
&
\DDDD\mathsf{1}\ar[r]_{\DDDD(d^0 ) }\ar@{ { } }@/_1ex/[rd]|-{\xRightarrow{\DDDD(n_0)} }\ar@{=}@/_5ex/[rd]
&
\DDDD\mathsf{2}\ar[d]^{\DDDD (s^0)}
\\
&
\DDDD\mathsf{1}
&
&
&
\DDDD\mathsf{1}
}$$ 
\end{itemize}
Moreover, there is a bijection between $2$-functors $\dot{\Delta }\to\aaa $ and $2$-functors $\overline{\DDDD } : \dot{\Delta } _ \ell\to \aaa $ such that $\overline{\DDDD } ( \vartheta ) $ is an invertible $2$-cell. 
\end{prop}

Let $\aaa $ be a $2$-category and $\DDDD : \Delta _ \ell \to \aaa $ be a pseudofunctor. If the weighted bilimit $\left\{ \dot{\Delta } (\mathsf{0}, \j -), \DDDD\right\} _ {\bi } $ exists, we say that  $\left\{ \dot{\Delta } (\mathsf{0}, \j -), \DDDD\right\} _ {\bi } $ is the \textit{descent object} of $\DDDD $. 
Moreover, if the weighted bilimit $\left\{ \dot{\Delta }_\ell (\mathsf{0}, \t -), \DDDD\right\} _ {\bi } $ exists, it is called the \textit{lax descent object} of $\DDDD $.

Analogously, if such $\DDDD$ is a $2$-functor and the (strict) weighted $2$-limit $\left\{ \dot{\Delta } (\mathsf{0}, \j -), \DDDD\right\} $ exists, we call it the \textit{strict descent object} of $\DDDD $. Finally, the (strict) weighted $2$-limit $\left\{ \dot{\Delta }_ \ell (\mathsf{0}, \t -), \DDDD\right\} $ is called the \textit{strict lax descent object} of $\DDDD $, if it exists.

\begin{lem}\label{strictequivalence}
Strict lax descent objects are lax descent objects and strict descent objects are descent objects. That is to say, the weights $\dot{\Delta } _ \ell (\mathsf{0}, \t -): \Delta _ \ell\to\Cat, \dot{\Delta }  (\mathsf{0}, \j -): \Delta _ \ell\to\Cat $ are flexible.
\end{lem}

The dual notions of lax descent object and descent object are called the codescent object and the lax codescent object. If $\AAAA : \Delta _ \ell ^{\op } \to \aaa $ is a $2$-functor, the \textit{codescent object} of $\AAAA $ is, if it exists, $\dot{\Delta } (\mathsf{0}, \j -)\Asterisk _ {\bi }\AAAA $ and the \textit{lax codescent object} of $\AAAA $ is $\dot{\Delta }_\ell (\mathsf{0}, \t -)\Asterisk _ {\bi }\AAAA $ if it exists.

Also, the weighted colimits $\dot{\Delta } (\mathsf{0}, \j -)\Asterisk \AAAA , \dot{\Delta } _\ell (\mathsf{0}, \t -)\Asterisk \AAAA $ are called, respectively, the \textit{strict codescent object }and the \textit{strict lax codescent object} of $\AAAA $.

\begin{rem}\label{strictdescentstrict}
If 
$\DDDD : \Delta _ {\ell }\to\Cat $
is a $2$-functor, then
$$\left\{ \dot{\Delta }_\ell (\mathsf{0}, \t -), \DDDD \right\} \cong \left[ \Delta _\ell , \Cat\right] \left( \dot{\Delta } _ \ell (\mathsf{0}, \t -), \DDDD \right) .$$
Thereby, we can describe the strict lax descent object of $\DDDD :\Delta _ \ell\to\Cat $ explicitly as follows:

\begin{enumerate}
\item Objects are $2$-natural transformations $\mathtt{f}: \dot{\Delta } _ \ell (\mathsf{0}, \t -)\longrightarrow \DDDD $. We have a bijective correspondence between such $2$-natural transformations and pairs $(f,\left\langle  \overline{\mathsf{f}}\right\rangle)$ in which $f$ is an object of $ \DDDD\mathsf{1} $ and $\left\langle  \overline{\mathsf{f}}\right\rangle: \DDDD (d^1)f\to\DDDD (d^0)f $ is a morphism in $ \DDDD\mathsf{2} $ satisfying the following equations:	
\begin{itemize}
\renewcommand\labelitemi{--}
\item Associativity:
\small
$$\left(\DDDD (\sigma _ {{}_{00}} )_ {{}_f}\right)\left(\DDDD (\partial ^0)(\left\langle  \overline{\mathsf{f}}\right\rangle)\right) \left( \DDDD (\sigma _ {{}_{20}}) _ {{}_{f}}\right)\left(\DDDD (\partial ^2)(\left\langle  \overline{\mathsf{f}}\right\rangle )\right) = \left(\DDDD(\partial ^1)(\left\langle  \overline{\mathsf{f}}\right\rangle)\right)\left(\DDDD (\sigma _ {{}_{21}}) _ {{}_{f}}\right);   $$
\normalsize
\item Identity:
\small
$$\left(\DDDD(s^0) (\left\langle  \overline{\mathsf{f}}\right\rangle) \right)\left(\DDDD(n_1) _ {{}_f}\right) = \left(\DDDD(n_0) _ {{}_f}\right) $$
\normalsize
\end{itemize}
If $\mathtt{f}: \dot{\Delta } _ \ell (\mathsf{0}, -)\longrightarrow \DDDD $ is a $2$-natural transformation, we get such pair by the correspondence 
$\mathtt{f}\mapsto (\mathtt{f} _ {{}_{\mathsf{1} }}(d), \mathtt{f} _ {{}_{\mathsf{2} }}(\vartheta )) $.

\item The morphisms are modifications. In other words, a morphism $\mathfrak{m} : \mathtt{f}\to\mathtt{h} $ is determined by a morphism $\mathfrak{m}: f\to g $ in $\DDDD \mathsf{1} $ such that $\DDDD (d^0)(\mathfrak{m} )\left\langle   \overline{\mathsf{f}}\right\rangle =  \left\langle \overline{\mathsf{h}}\right\rangle\DDDD (d^1)(\mathfrak{m} ) $.
\end{enumerate}

Furthermore, there is a full inclusion 
$\left\{ \dot{\Delta } (\mathsf{0}, \j -), \DDDD \right\} \to\left\{ \dot{\Delta }_\ell (\mathsf{0}, \t -), \DDDD \right\}   $
such that the objects of  $\left\{ \dot{\Delta } (\mathsf{0}, \j -), \DDDD \right\}$ are precisely the pairs $(f,\left\langle  \overline{\mathsf{f}}\right\rangle)$ (described above) with one further property:  $\left\langle\overline{\mathsf{f}}\right\rangle $ is actually an isomorphism in  $ \DDDD\mathsf{2} $.
\end{rem}

\section{Pseudomonads and lax algebras}\label{Laxalgebras}
Pseudomonads in $2$-$\Cat $ are defined in \cite{Lucatelli1, Lucatelli2}. The definition agrees with the theory of pseudomonads for $\mathsf{Gray}$-categories~\cite{MARMOLEJOD, SLACK2000, MARMOLEJO, MARMOLEJOR} and with the definition of \textit{doctrines} of \cite{RS80}.

For each pseudomonad $\TTTTT$ on a $2$-category $\bbb $, there is an associated (right biadjoint) forgetful $2$-functor 
$\mathsf{Ps}\textrm{-}\TTTTT\textrm{-}\Alg\to \bbb $, 
in which $\mathsf{Ps}$-$\TTTTT\textrm{-}\Alg $ is the $2$-category of pseudoalgebras. In this section, we give the definitions of the $2$-category of lax algebras $\mathsf{Lax}\textrm{-}\TTTTT\textrm{-}\Alg _\ell $ and its associated forgetful $2$-functor $ \mathsf{Lax}\textrm{-}\TTTTT\textrm{-}\Alg _\ell\to\bbb $, which are slight generalizations of the definitions given in \cite{RSY, SLACK2020}.

Recall that a \textit{pseudomonad} $\TTTTT$ on a $2$-category $\bbb $ consists of a sextuple $(\TTTTT,  m  , \eta , \mu,   \iota, \tau )$, in which $\TTTTT:\bbb\to\bbb $ is a pseudofunctor, $ m  : \TTTTT^2\longrightarrow \TTTTT,  \eta : \Id _ {{}_\bbb }\longrightarrow \TTTTT$ are pseudonatural transformations and
$\tau : \Id _ {{}_{\TTTTT}}\Longrightarrow (m)(\TTTTT\eta ) $, $\iota : (m)(\eta\TTTTT )\Longrightarrow \Id _ {{}_{\TTTTT}}$, $\mu : m\left(\TTTTT m\right)\Rightarrow m\left( m\TTTTT \right)  $
are invertible modifications satisfying the following coherence equations:
\begin{itemize}
\renewcommand\labelitemi{--}
\item Associativity:
$$\xymatrix{ \TTTTT^4\ar[r]^-{\TTTTT^2  m  }\ar[dr]|-{\TTTTT m \TTTTT}\ar[d]_{ m \TTTTT^2}
&
\TTTTT^3\ar[dr]^-{\TTTTT m }\ar@{}[d]|-{\xLeftarrow{\widehat{\TTTTT\mu} }}
&
&&
\TTTTT^4\ar[r]^-{\TTTTT^2 m }\ar@{}[dr]|-{\xLeftarrow{ m  _ {{}_{{}_{ m  } }}^{-1}}}\ar[d]_-{ m   \TTTTT^2}
&
\TTTTT^3\ar[d]|-{ m  \TTTTT}\ar[dr]^-{\TTTTT m  }
&
\\
\TTTTT^3\ar[dr]_-{ m \TTTTT}\ar@{}[r]|{\xLeftarrow{\mu \TTTTT}}
&
\TTTTT^3\ar[r]|-{\TTTTT m  }\ar[d]|-{ m \TTTTT}\ar@{}[dr]|-{ \xLeftarrow{\hskip 0.1cm\mu \hskip 0.1cm} }
&
\TTTTT^2\ar[d]^-{ m  }
&=&
\TTTTT^3\ar[r]|-{\TTTTT m  }\ar[dr]_-{ m  \TTTTT}
&
\TTTTT^2\ar@{}[r]|-{\xLeftarrow{\hskip 0.1cm \mu\hskip 0.1cm }}\ar[dr]|-{m}\ar@{}[d]|-{\xLeftarrow{\hskip 0.1cm\mu\hskip 0.1cm } }
&
\TTTTT^2\ar[d]^{ m  }
\\
&
\TTTTT^2\ar[r]_{ m }
&
\TTTTT
&&
&
\TTTTT^2\ar[r]_ { m  }
&
\TTTTT
}$$ 
\item Identity:
$$\xymatrix{ &
\TTTTT^2\ar[dl]_-{\TTTTT\eta \TTTTT}\ar[dr]^-{\TTTTT\eta\TTTTT}\ar[dd]|-{\Id _ {{}_{\TTTTT^2}}}
&
&&
&
\TTTTT^2\ar[d]|-{\TTTTT\eta\TTTTT}
&
\\
\TTTTT^3\ar[dr]_-{ m \TTTTT}\ar@{}[r]|-{\xLeftarrow{\tau\TTTTT} }
&
&
\TTTTT^3\ar@{}[l]|-{\xLeftarrow{\widehat{\TTTTT\iota}}}\ar[dl]^-{\TTTTT  m  }
&&
&\TTTTT^3\ar[dl]|-{ m \TTTTT}\ar[dr]|-{\TTTTT m  }
&
\\
&
\TTTTT^2\ar[d]|-{ m  }                 
&
&=&
\TTTTT^2\ar@{}[rr]|-{\xLeftarrow{\hskip 0.2cm \mu\hskip 0.2cm } }\ar[dr]|-{ m  }
&&
\TTTTT ^2\ar[dl]|-{ m  }
\\
&
\TTTTT
&
&&
&
\TTTTT
&
}$$ 
\end{itemize}
in which
$$\widehat{\TTTTT\iota } := \left( \tttt _ {{}_{\TTTTT}} \right)^{-1} \left(\TTTTT\iota \right) \left(\tttt _ {{}_{( m )(\eta\TTTTT ) }}\right)\qquad\qquad
\widehat{\TTTTT\mu }  :=  \left( \tttt _{{}_{( m ) ( m \TTTTT ) }}\right)^{-1}\left(\TTTTT\mu\right) \left( \tttt _{{}_{( m ) (\TTTTT m  ) }}\right). $$
Recall that the $\Cat$-enriched notion of monad is a pseudomonad $\TTTTT = (\TTTTT,  m  , \eta , \mu,   \iota, \tau )$ such that the invertible modifications $\mu,   \iota, \tau $ are identities and $m, \eta $ are $2$-natural transformations. In this case, we say that $\TTTTT = (\TTTTT,  m  , \eta )$ is a $2$-monad, omitting the identities.

\begin{defi}[Lax algebras]\label{LAXALGEBRAS}
Let $\TTTTT = (\TTTTT,  m  , \eta , \mu,   \iota, \tau )$ be a pseudomonad on $\bbb $. We define the $2$-category $\mathsf{Lax}\textrm{-}\TTTTT\textrm{-}\Alg _\ell $
as follows:
\begin{enumerate}
\item Objects: \textit{lax $\TTTTT$-algebras} are defined by $\mathtt{z}= (Z, \alg _ {{}_{\mathtt{z}}}, \overline{{\underline{\mathtt{z}}}}, \overline{{\underline{\mathtt{z}}}}_0 )$ in which $\alg _ {{}_{\mathtt{z}}}: \TTTTT Z\to Z $
is a morphism of $\bbb $ and $\overline{{\underline{\mathtt{z}}}}:\alg _ {{}_{\mathtt{z}}}\TTTTT(\alg _ {{}_{\mathtt{z}}})\Rightarrow \alg _ {{}_{\mathtt{z}}}m_{{}_{Z}}, \overline{{\underline{\mathtt{z}}}}_0: \Id _ {{}_{Z}}\Rightarrow\alg _ {{}_{\mathtt{z}}}\eta _ {{}_{Z}}  $ are $2$-cells of $\bbb $ satisfying the coherence axioms:

$$\xymatrix@C=4em{ 
\TTTTT ^3 Z\ar[r]^{\TTTTT ^2 (\alg _ {{}_{\mathtt{z}}} ) }\ar[rd]|-{\TTTTT (m _ {{}_{Z}}) }\ar[d]_{m_{{}_{\TTTTT Z}} }
&
\TTTTT ^2 Z\ar@{}[d]|{\xLeftarrow{\widehat{\TTTTT (\overline{{\underline{\mathtt{z}}}} )} }}\ar[rd]^{\TTTTT (\alg _ {{}_{\mathtt{z}}} )}
&\ar@{}[rdd]|{=}
&
\TTTTT ^3 Z\ar[r]^{\TTTTT ^2(\alg _ {{}_{\mathtt{z}}}) }\ar@{}[rd]|{\xLeftarrow{m _ {{}_{\alg _ {{}_{\mathtt{z}}} }}^{-1}  }}\ar[d]_{m _ {{}_{\TTTTT Z }} }
&
\TTTTT^2 Z\ar[d] ^{m _ {{}_{Z}} }\ar[rd]^{\TTTTT (\alg _ {{}_{\mathtt{z}}}) }
&
\\
\TTTTT^2 Z\ar@{}[r]|-{\xLeftarrow{\mu _ {{}_{ Z }}  }}\ar[rd]_-{m _ {{}_{Z}} }
&
\TTTTT ^2 Z\ar[d]_-{m _ {{}_{Z}} }\ar[r]^{\TTTTT (\alg _ {{}_{\mathtt{z}}} ) }\ar@{}[rd]|{\xLeftarrow{\overline{{\underline{\mathtt{z}}} }  }}
&
\TTTTT Z\ar[d]^{\alg _ {{}_{\mathtt{z}}} }
&
\TTTTT ^2 Z\ar[r]_-{\TTTTT (\alg _ {{}_{\mathtt{z}}} ) }\ar[rd]_-{m _ {{}_{Z}} }
&
\TTTTT Z\ar@{}[d]|{\xLeftarrow{\overline{{\underline{\mathtt{z}}} }  } } \ar@{}[r]|{\xLeftarrow{\overline{{\underline{\mathtt{z}}} }}  }\ar[rd]|-{\alg _ {{}_{\mathtt{z}}} }
&
\TTTTT Z\ar[d]^{\alg _ {{}_{\mathtt{z}}} }
\\
&
\TTTTT Z\ar[r]_ {\alg _ {{}_{\mathtt{z}}} }
&
Z
&
&
\TTTTT Z\ar[r]_ {\alg _ {{}_{\mathtt{z}}} }
&
Z
}$$
in which 
$\widehat{\TTTTT (\overline{{\underline{\mathtt{z}}}} )}:= \left( \tttt _{{}_{( \alg _ {{}_{\mathtt{z}}})  ( m _{{}_{Z}} ) }}\right)^{-1}\left(\TTTTT (\overline{{\underline{\mathtt{z}}}})\right) \left( \tttt _{{}_{( \alg _ {{}_{\mathtt{z}}} ) (\TTTTT (\alg _ {{}_{\mathtt{z}}})  ) }}\right)$ and the $2$-cells
$$\xymatrix{ 
\TTTTT Z\ar[rrr]^{\alg _ {{}_{\mathtt{z}}}}\ar[rd]^{\eta _ {{}_{\TTTTT Z}} }\ar@{=}[dd]
&
&
&
Z\ar[ld]_-{\eta _ {{}_{Z}} }\ar@{=}[dd]
&
\TTTTT Z\ar@{=}[rr]\ar@{=}[dd]\ar[rd]^{\TTTTT (\eta _ {{}_ {Z}} )}
&
&
\TTTTT Z\ar@{}[ld]|-{\xLeftarrow{\widehat{\TTTTT (\overline{{\underline{\mathtt{z}}}}_0) }} }\ar@{=}[d]
\\
\ar@{}[r]|-{\xLeftarrow{\iota _ {{}_{Z}} } }
&
\TTTTT ^2 Z\ar@{}[ru]|-{\xLeftarrow{\eta _ {{}_{\alg _ {{}_{\mathtt{z}}}}}^{-1} } }\ar[r]_{\TTTTT (\alg _ {{}_{\mathtt{z}}} ) }\ar[ld]^{m _ {{}_{Z}} }
\ar@{}[rd]|-{\xLeftarrow {\overline{{\underline{\mathtt{z}}} } } }
&
\TTTTT Z\ar@{}[r]|-{\xLeftarrow{\overline{{\underline{\mathtt{z}}} }_0 } }\ar[rd]_-{\alg _ {{}_{\mathtt{z}}} }
&
&
\ar@{}[r]|-{\xLeftarrow{\tau _ {{}_{Z}}^{-1} } }
&
\TTTTT ^2 Z\ar[r]_-{\TTTTT (\alg _ {{}_{\mathtt{z}}}) }\ar[ld]^- {m _{{}_{Z}} }\ar@{}[rd]|-{\xLeftarrow{\overline{{\underline{\mathtt{z}}} }} }
&
\TTTTT Z\ar[d]^-{\alg _ {{}_{\mathtt{z}}}}
\\
\TTTTT Z\ar[rrr]_{\alg _ {{}_{\mathtt{z}}} }
&
&
&
Z
&
\TTTTT Z\ar[rr]_-{\alg _ {{}_{\mathtt{z}}}}
&&
Z
}$$
are identities in which $\widehat{\TTTTT (\overline{{\underline{\mathtt{z}}}}_0 )} := \left(\tttt _ {{}_{( \alg _ {{}_{\mathtt{z}}} )(\eta _ {{}_{Z}} ) }}\right)^{-1} \left(\TTTTT (\overline{{\underline{\mathtt{z}}}}_0)  \right) 
\left( \tttt _ {{}_{\TTTTT Z}} \right)$.  
Recall that, if a lax algebra $\mathtt{z}= (Z, \alg _ {{}_{\mathtt{z}}}, \overline{{\underline{\mathtt{z}}}}, \overline{{\underline{\mathtt{z}}}}_0 )$ is such that $\overline{{\underline{\mathtt{z}}}}, \overline{{\underline{\mathtt{z}}}}_0 $ are invertible $2$-cells, then $\mathtt{z}$ is called a pseudoalgebra.

\item Morphisms: \textit{lax $\TTTTT$-morphisms} $\mathtt{f}:\mathtt{y}\to\mathtt{z} $ between lax $\TTTTT$-algebras $\mathtt{y}=(Y, \alg _ {{}_{\mathtt{y}}}, \overline{{\underline{\mathtt{y}}}}, \overline{{\underline{\mathtt{y}}}}_0 )$, $\mathtt{z}= (Z, \alg _ {{}_{\mathtt{z}}}, \overline{{\underline{\mathtt{z}}}}, \overline{{\underline{\mathtt{z}}}}_0 )$ are pairs $\mathtt{f} = (f, \left\langle \overline{\mathsf{f}}\right\rangle ) $ in which 
$f: Y\to Z  $
is a morphism in $\bbb $ and
$\left\langle   \overline{\mathsf{f}}\right\rangle :  \alg _ {{}_{\mathtt{z}}}\TTTTT (f) \Rightarrow f\alg _ {{}_{\mathtt{y}}}   $  
is a $2$-cell of $\bbb $ such that, defining $\widehat{\TTTTT (\left\langle   \overline{\mathsf{f}}\right\rangle )} : = \tttt ^{-1} _ {{}_{(f)(\alg _ {{}_{\mathtt{y}}}) }}  \TTTTT (\left\langle   \overline{\mathsf{f}}\right\rangle )\tttt _ {{}_{(\alg _ {{}_{\mathtt{z}}})(\TTTTT(f))  }}$, the equations

$$\xymatrix@C=2.4em{
&
\TTTTT ^2 Y\ar@{}[d]|-{\xLeftarrow{m_{{}_{f}}^{-1} } }\ar[ld]_ -{m _ {{}_{Y}} }\ar[r]^-{\TTTTT ^2(f) }
&
\TTTTT ^2 Z\ar[rd]^-{\TTTTT (\alg _ {{}_{\mathtt{z}}}) }\ar[ld]|-{m _ {{}_{Z}} }\ar@{}[dd]|-{\xLeftarrow{\,\,\overline{{\underline{\mathtt{z}}} }\, \,} }
&
\ar@{}[rdd]|-{=}
&
&
\TTTTT Z\ar[r]^{\alg _ {{}_{\mathtt{z}}}}\ar@{}[dd]|-{\xRightarrow{\widehat{\TTTTT (\left\langle   \overline{\mathsf{f}}\right\rangle)}}}
&
Z\ar@{}[d]|-{\xRightarrow{\left\langle\overline{\mathsf{f}}\right\rangle } }
&
\\
\TTTTT Y\ar[r]_{\TTTTT (f) }\ar[rd]_{\alg _ {{}_{\mathtt{y}}}}
&
\TTTTT Z\ar[rd]|-{\alg _ {{}_{\mathtt{z}}} }\ar@{}[d]|-{\xLeftarrow{\left\langle   \overline{\mathsf{f}}\right\rangle} }
&
&
\TTTTT Z\ar[ld]^-{\alg _ {{}_{\mathtt{z}}} }
&
\TTTTT ^2 Z\ar[ru]^-{\TTTTT (\alg _ {{}_{\mathtt{z}}} )}
&
&
\TTTTT Y\ar[lu]|-{\TTTTT (f) }\ar[r]^{\alg _ {{}_{\mathtt{y}}} }\ar@{}[d]|-{\xRightarrow{\,\overline{{\underline{\mathtt{y}}} }\, } }
&
Y\ar[lu]_ {f}
\\
&
Y\ar[r]_{f}
&
Z
&
&
&
\TTTTT ^2Y\ar[ru]|-{\TTTTT (\alg _ {{}_{\mathtt{y}}})}\ar[r]_{m _ {{}_{Y}} }\ar[lu]^{\TTTTT ^2(f)}
&
\TTTTT Y\ar[ru]_-{\alg _ {{}_{\mathtt{y}}} }
&
}$$
$$\xymatrix{ 
Y\ar[rr]^{f}\ar[d]_{\eta _ {{}_{Y}} }\ar@{}[rd]|-{\xLeftarrow{\eta _ {{}_{f}}^{-1}} }
&
&
Z\ar[ld]_-{\eta _ {{}_{Z}} }\ar@{=}[dd]\ar@{}[rdd]|-{=}
&
&
Y\ar[ld]_{\eta _ {{}_{Y}} }\ar@{=}[dd]
&
\\
\TTTTT Y\ar[r]_{\TTTTT (f)}\ar[d]_-{\alg _ {{}_{\mathtt{y}}} }\ar@{}[rd]|-{\xLeftarrow{\left\langle   \overline{\mathsf{f}}\right\rangle}}
&
\TTTTT Z\ar@{}[r]|-{\xLeftarrow{\overline{{\underline{\mathtt{z}}}}_0} }\ar[rd]_-{\alg _ {{}_{\mathtt{z}}} }
&
&
\TTTTT Y\ar@{}[r]|-{\xLeftarrow{\overline{{\underline{\mathtt{y}}}}_0 } }\ar[rd]_{\alg _ {{}_{\mathtt{y}}}}
&
&
\\
Y\ar[rr]_-{f}
&
&
Z
&
&
Y\ar[r]_{f}
&
Z
}$$
hold. Recall that a lax $\TTTTT $-morphism $\mathtt{f} = (f,  \left\langle \overline{\mathsf{f}}\right\rangle ) $ is called a $\TTTTT $-pseudomorphism if $\left\langle   \overline{\mathsf{f}}\right\rangle$ is an invertible $2$-cell. If $\left\langle \overline{\mathsf{f}}\right\rangle $ is an identity, $\mathtt{f}$ is called a (strict) $\TTTTT $-morphism.

\item $2$-cells: a \textit{$\TTTTT $-transformation}  $\mathfrak{m} : \mathtt{f}\Rightarrow\mathtt{h} $ between lax $\TTTTT $-morphisms $\mathtt{f} = (f,\left\langle   \overline{\mathsf{f}}\right\rangle )$, $\mathtt{h} = (h, \left\langle\overline{\mathsf{h}}\right\rangle )$ is a $2$-cell $\mathfrak{m} : f\Rightarrow h $ in $\bbb $ such that  the equation below holds.

$$\xymatrix{  \TTTTT Y\ar@/_4ex/[dd]_{\TTTTT (f) }
                    \ar@{}[dd]|{\xRightarrow{\TTTTT(\mathfrak{m}) } }
                    \ar@/^4ex/[dd]^{\TTTTT (h) }
										\ar[rr]^{\alg _ {{}_{\mathtt{y}}} } && 
										 Y\ar[dd]^{h }    
&&
\TTTTT Y \ar[rr]^{ \alg _ {{}_{\mathtt{y}}} }\ar[dd]_-{\TTTTT (f) }  &&
Y\ar@/_5ex/[dd]_{f}
                    \ar@{}[dd]|{\xRightarrow{\mathfrak{m}  } }
                    \ar@/^5ex/[dd]^{h }
\\
&\ar@{}[r]|{\xRightarrow{\hskip .2em \left\langle\overline{\mathsf{h}}\right\rangle  \hskip .2em } }   &
 &=& 
&\ar@{}[l]|{\xRightarrow{\hskip .2em \left\langle \overline{\mathsf{f}}\right\rangle  \hskip .2em } }  & 
\\
 \TTTTT Z\ar[rr]_ {\alg _ {{}_{\mathtt{z}}} } &&  Z
 &&
\TTTTT Z\ar[rr]_ {\alg _ {{}_{\mathtt{z}}} } && Z	 }$$

\end{enumerate}

The compositions are defined in the obvious way and these definitions make $\mathsf{Lax}\textrm{-}\TTTTT\textrm{-}\Alg _ {\ell } $ a $2$-category. The full sub-$2$-category of the pseudoalgebras of $\mathsf{Lax}\textrm{-}\TTTTT\textrm{-}\Alg _ {\ell } $ is denoted by $\mathsf{Ps}\textrm{-}\TTTTT\textrm{-}\Alg _ {\ell } $. Also, the locally full sub-$2$-category  consisting of lax algebras and pseudomorphisms between them is denoted by
$\mathsf{Lax}\textrm{-}\TTTTT\textrm{-}\Alg  $. Finally, the full sub-$2$-category of the pseudoalgebras of $\mathsf{Lax}\textrm{-}\TTTTT\textrm{-}\Alg  $ is denoted by  $\mathsf{Ps}\textrm{-}\TTTTT\textrm{-}\Alg $. In short, we have locally full inclusions:
$$\xymatrix{
\mathsf{Ps}\textrm{-}\TTTTT\textrm{-}\Alg\ar[r]\ar[d]
&
\mathsf{Ps}\textrm{-}\TTTTT\textrm{-}\Alg _ {\ell }\ar[d]
\\
\mathsf{Lax}\textrm{-}\TTTTT\textrm{-}\Alg \ar[r]|-{\ell }
&
\mathsf{Lax}\textrm{-}\TTTTT\textrm{-}\Alg  _ {\ell }
}$$

\end{defi}

\begin{rem}\label{coherenceinclusions}
If $\TTTTT = (\TTTTT,  m  , \eta ) $ is a $2$-monad, we denote by $\TTTTT\textrm{-}\Alg _{\ell }$ the full sub-$2$-category of \textit{strict algebras} of $\mathsf{Lax}\textrm{-}\TTTTT\textrm{-}\Alg _ {\ell } $. That is to say, the objects of   $\TTTTT\textrm{-}\Alg _{\ell }$ are the lax $\TTTTT $-algebras $\mathtt{y}=(Y, \alg _ {{}_{\mathtt{y}}}, \overline{{\underline{\mathtt{y}}}}, \overline{{\underline{\mathtt{y}}}}_0 )$ such that its $2$-cells $\overline{{\underline{\mathtt{y}}}}, \overline{{\underline{\mathtt{y}}}}_0 $ are identities. 

Also, we denote by $\TTTTT\textrm{-}\Alg$ the locally full sub-$2$-category of  $\TTTTT\textrm{-}\Alg _{\ell }$ consisting of strict algebras and pseudomorphisms between them. Finally, $\TTTTT\textrm{-}\Alg _ {\textrm{s} }$ is the locally full sub-$2$-category $\TTTTT\textrm{-}\Alg _{\ell }$ consisting of strict algebras and \textit{strict morphisms} between them. That is to say, the $1$-cells of   $\TTTTT\textrm{-}\Alg _ {\textrm{s} }$ are the pseudomorphisms $\mathtt{f} = (f,  \left\langle \overline{\mathsf{f}}\right\rangle ) $ such that $\left\langle   \overline{\mathsf{f}}\right\rangle$ is the identity. In this case, we have locally full inclusions
$$\xymatrix{\TTTTT\textrm{-}\Alg _ {\textrm{s} }\ar[r]\ar[d]
&
\TTTTT\textrm{-}\Alg \ar[r]\ar[d]
&
\TTTTT\textrm{-}\Alg _ {\ell }\ar[d]
\\
\mathsf{Ps}\textrm{-}\TTTTT\textrm{-}\Alg_{\textrm{s}}\ar[d]\ar[r]
&
\mathsf{Ps}\textrm{-}\TTTTT\textrm{-}\Alg\ar[r]\ar[d]
&
\mathsf{Ps}\textrm{-}\TTTTT\textrm{-}\Alg _ {\ell }\ar[d]
\\
\mathsf{Lax}\textrm{-}\TTTTT\textrm{-}\Alg _{\textrm{s}}\ar[r]
&
\mathsf{Lax}\textrm{-}\TTTTT\textrm{-}\Alg \ar[r]|-{\ell }
&
\mathsf{Lax}\textrm{-}\TTTTT\textrm{-}\Alg  _ {\ell }
}$$
in which the vertical arrows are full. 

\end{rem}

\begin{rem}
There is a vast literature of examples of pseudomonads, $2$-monads and their respective algebras, pseudoagebras and lax algebras~\cite{KOCK, Power, Hermida, Power2000}. The reader can keep in mind three very simple examples:

\begin{itemize}
\renewcommand\labelitemi{--}
\item The ``free $2$-monad'' $\TTTTT $ on $\Cat $ whose pseudoalgebras are unbiased monoidal categories. This is defined by $\TTTTT X:= \displaystyle\coprod_{n=0} ^\infty X^n $, in which $X^{n+1}:=X^n\times X $ and $X^0:= I$ is the terminal category, with the obvious pseudomonad structure. In this case, the $\TTTTT $-pseudomorphisms are the so called strong monoidal functors, while the lax $\TTTTT $-morphisms are the lax monoidal functors~\cite{Leinster}. 

\item The most simple example is the pseudomonad rising from a monoidal category.
A monoidal category $M$ is just a pseudomonoid~\cite{RSP} of $\Cat $ and, therefore, it gives rise to a pseudomonad $\TTTTT : \Cat\to \Cat $ defined by $\TTTTT X = M\times X $ with obvious unit and multiplication (and invertible modifications) coming from the monoidal structure of $M$. 
The pseudoalgebras and lax algebras of this pseudomonad are called, respectively, the pseudoactions and lax actions of $M$. Lax actions of a monoidal category $M$ are also called \textit{graded monads} (see \cite{gradedmonads}).

The inclusion $\Set\to\Cat $ is a strong monoidal functor w.r.t. the cartesian structures, since this functor preserves products. In particular, it takes monoids of $\Set $ to monoids of $\Cat $. In short, this means that we can see a monoid $M$ as a (discrete) strict monoidal category. Therefore, a monoid $M$ gives rise to a $2$-monad $\TTTTT X = M\times X $ as defined above. In this case, the $2$-categories $\TTTTT \textrm{-}\Alg _ {\textrm{s}}$, $\mathsf{Ps}\textrm{-}\TTTTT \textrm{-}\Alg$ and $\mathsf{Lax}\textrm{-}\TTTTT\textrm{-}\Alg  _ {\ell }$ are, respectively, the $2$-categories of (strict) actions, pseudoactions (as defined in \cite{D}) and lax actions of this monoid $M$ on categories. A lax action of the trivial monoid on a category is the same as a monad.

\item Let $\sss $ be a small $2$-category and $\aaa $ a $2$-category. We denote by $\sss _ 0 $ the discrete $2$-category of the objects of $\sss $ and by $\left[ \sss , \aaa\right]$ the $2$-category of $2$-functors, $2$-natural transformations and modifications. If the restriction $\left[ \sss , \aaa\right] \to \left[ \sss _ 0, \aaa\right] $ has a left $2$-adjoint (called the global left Kan extension), then the restriction is $2$-monadic and  $\left[ \sss , \aaa\right] _ {PS}$ is the $2$-category of $\TTTTT $-pseudoalgebras (in which $\TTTTT $ is the $2$-monad induced by the $2$-adjunction). Also, the $2$-category of lax algebras is the $2$-category $\left[  \sss , \Cat\right] _ {Lax} $ of lax functors $\sss\to\aaa $, lax natural transformations and modifications~\cite{Power}.

Again, if $M$ is a monoid (of $\Set$), $M$ can be seen as a category with only one object~\cite{Leinster2}, usually denoted by $\sum M $. That is to say, the locally discrete $2$-category $\sum M$ has only one object $\ast $ and $\sum M ( \ast , \ast ) : = M $ is the discrete category with the composition of $1$-cells given by the product of the monoid. In this case, the restriction 
$$\left[  \sum M , \Cat\right]\to \left[  	\left( \sum M\right) _0 , \Cat\right]\cong \Cat $$
has a left $2$-adjoint (and, as explained, it is $2$-monadic). The left $2$-adjoint is given by $X\mapsto \Lan _ {{}_{\left( \sum M\right) _0\to \sum M }} X $ in which
\begin{eqnarray*}
\Lan _ {{}_{\left( \sum M\right) _0\to \sum M }} X: &\sum M &\to\Cat\\
&\ast &\to M\times X\\
&M\ni g&\mapsto \overline{g} : (h,x)\mapsto (gh,x).
\end{eqnarray*}  
This $2$-adjunction is precisely the same $2$-adjunction between strict $\TTTTT $-algebras and the base $2$-category $\Cat $, if $\TTTTT $ is the $2$-monad $\TTTTT X = M\times X $ described above. Hence the $2$-category of pseudoalgebras $\left[  \sum M , \Cat\right] _ {PS}$ and the $2$-category $\left[  \sum M , \Cat\right] _ {Lax} $
are, respectively, isomorphic to the $2$-category of pseudoactions and the $2$-category of lax actions of $M$ on categories. Moreover, $\TTTTT $-$\Alg _ {\textrm{s}}\to\Cat $ is $2$-comonadic.

More generally, if $M$ is a monoidal category, $M$ can be seen as a bicategory with only one object (see \cite{Leinster, BE}), also denoted by $\sum M $. The restriction $2$-functor $\left[  \sum M , \Cat\right] _ {PS}\to \left[  	\left( \sum M\right) _0 , \Cat\right] _ {PS}\cong \Cat $ is pseudomonadic and pseudocomonadic. Furthermore, it coincides with the forgetful pseudofunctor $\mathsf{Ps}\textrm{-}\TTTTT\textrm{-}\Alg\to\Cat $ in which $\TTTTT X = M\times X $ is given by the structure of the monoidal category (as above).

\end{itemize}  
\end{rem}

\begin{rem}\label{forgetful}
Let $\TTTTT = (\TTTTT,  m  , \eta , \mu,   \iota, \tau )$ be a pseudomonad on a $2$-category $\bbb $. If $\ccc $ is any sub-$2$-category of $\mathsf{Lax}\textrm{-}\TTTTT\textrm{-}\Alg $, we have a forgetful $2$-functor
\begin{eqnarray*}
U: &\ccc & \to \bbb\\ 
&\mathtt{z}= (Z, \alg _ {{}_{\mathtt{z}}}, \overline{{\underline{\mathtt{z}}}}, \overline{{\underline{\mathtt{z}}}}_0 )&\mapsto Z\\
&\mathtt{f} = (f,  \left\langle \overline{\mathsf{f}}\right\rangle )&\mapsto f\\
&\mathfrak{m} &\mapsto \mathfrak{m}
\end{eqnarray*}
\end{rem}

\begin{prop}\label{laxeffectivedescent}
Let $\TTTTT = (\TTTTT,  m  , \eta , \mu,   \iota, \tau )$ be a pseudomonad on a $2$-category $\bbb $. Given lax $\TTTTT$-algebras 
$\mathtt{y}=(Y, \alg _ {{}_{\mathtt{y}}}, \overline{{\underline{\mathtt{y}}}}, \overline{{\underline{\mathtt{y}}}}_0 )$, $\mathtt{z}= (Z, \alg _ {{}_{\mathtt{z}}}, \overline{{\underline{\mathtt{z}}}}, \overline{{\underline{\mathtt{z}}}}_0 )$
the category $\mathsf{Lax}\textrm{-}\TTTTT\textrm{-}\Alg _\ell (\mathtt{y}, \mathtt{z}) $ is the strict lax descent object  of the diagram
$\mathbb{T}_{\mathtt{z}}^{\mathtt{y}}: \Delta _ \ell \to \Cat $
\begin{equation*}\tag{$\mathbb{T}_{\mathtt{z}}^{\mathtt{y}}$}
\xymatrix{  \bbb (U\mathtt{y}, U\mathtt{z})\ar@<-2.5ex>[rrr]_-{\bbb(\alg _ {{}_{\mathtt{y}}}, U\mathtt{z} )}\ar@<2.5ex>[rrr]^-{\bbb(\TTTTT U\mathtt{y}, \alg _ {{}_{\mathtt{z}}} )\circ \hspace{0.2em}\TTTTT _ {{}_{(U\mathtt{y},U\mathtt{z})}} } &&& \bbb (\TTTTT U\mathtt{y},  U\mathtt{z} )\ar[lll]|-{\bbb (\eta _ {{}_{U\mathtt{y}}} , U\mathtt{z})}
\ar@<-2.5 ex>[rrr]_-{\bbb(\TTTTT (\alg _ {{}_{\mathtt{y}}}), U\mathtt{z} )}\ar[rrr]|-{\bbb (m _{{}_{U\mathtt{y}}}, U\mathtt{z} )}\ar@<2.5ex>[rrr]^-{\bbb(\TTTTT ^2 U\mathtt{y}, \alg _ {{}_{\mathtt{z}}} )\circ \hspace{0.2em}\TTTTT _ {{}_{(\TTTTT U\mathtt{y},U\mathtt{z})}} } &&& \bbb (\TTTTT ^2 U \mathtt{y}, U\mathtt{z} ) }
\end{equation*}
such that
\small
\begin{eqnarray*}
\begin{aligned}
&\mathbb{T}_{\mathtt{z}}^{\mathtt{y}}(\sigma _ {20}) _{{}_{f}} &:=& \left( \id _ {{}_{  \alg _ {{}_{\mathtt{z}}} }}\ast \tttt _ {{}_{(f)(\alg _ {{}_{\mathtt{y}}}) }}  \right)   \\
&\mathbb{T}_{\mathtt{z}}^{\mathtt{y}}(\sigma _ {21}) _{{}_{f}} &:=&  \left( \id_ {{}_{f}}\ast \overline{{\underline{\mathtt{y}}} }\right)\\
&\mathbb{T}_{\mathtt{z}}^{\mathtt{y}}(n _ {1}) _{{}_{f}} &:=& \left( \id _ {{}_{f}}\ast \overline{{\underline{\mathtt{y}}}}_0 \right)
\end{aligned}
\begin{aligned}
&\mathbb{T}_{\mathtt{z}}^{\mathtt{y}}(\sigma _ {00}) _{{}_{f}} &:=& \left( \id _ {{}_{ \alg _ {{}_{\mathtt{z}}}  }}\ast m _ {{}_{f}}^{-1}\right)\cdot\left( \overline{{\underline{\mathtt{z}}}}\ast\id _ {{}_{\TTTTT ^2(f)  }}\right)\cdot
\left(\id _ {{}_{\alg _ {{}_{\mathtt{z}}} }}\ast \tttt _ {{}_{(\alg _ {{}_{\mathtt{z}}}) (\TTTTT(f)) }}^{-1} \right) \\
&\mathbb{T}_{\mathtt{z}}^{\mathtt{y}}(n _ {0}) _{{}_{f}} &:=&  
\left(\id _{{}_{\alg _ {{}_{\mathtt{z}}}}}\ast\eta _ {{}_{f}}^{-1} \right)\cdot \left(\overline{{\underline{\mathtt{z}}}}_0\ast\id _ {{}_{f}}\right)
\end{aligned}
\end{eqnarray*}
\normalsize
Furthermore, the strict descent object of $\mathbb{T}_{\mathtt{z}}^{\mathtt{y}}$ is $\mathsf{Lax}\textrm{-}\TTTTT\textrm{-}\Alg  (\mathtt{y}, \mathtt{z}) $.
\end{prop}
\begin{proof}
It follows from Definition \ref{LAXALGEBRAS} and Remark \ref{strictdescentstrict}.
\end{proof}

\begin{rem}\label{indeedpseudofunctor}
In the context of the proposition above, we can define a pseudofunctor $\mathbb{T}^{\mathtt{y}}: \Delta _ \ell\times \mathsf{Lax}\textrm{-}\TTTTT\textrm{-}\Alg\to\Cat $ in which $\mathbb{T}^{\mathtt{y}}(-,{\mathtt{z}}) : = \mathbb{T}^{\mathtt{y}}_ {\mathtt{z}}$, since the morphisms defined above are actually pseudonatural in $\mathtt{z}$ w.r.t. $\TTTTT $-pseudomorphisms and $\TTTTT $-transformations.
\end{rem}

Assume that the triangles below are commutative, $R$ is a right biadjoint pseudofunctor and the arrows without labels are the forgetful $2$-functors of Remark \ref{forgetful}. By Corollary \ref{used}, it follows from Proposition \ref{laxeffectivedescent} (and last remark) that, whenever $\aaa $ has lax codescent objects, $\ell\circ  J$ is right biadjoint to a pseudofunctor $G$. Also, for each lax algebra $\mathtt{y}$, there is a diagram $\AAAA _ \mathtt{y}$ such that  $G\mathtt{y}\simeq \dot{\Delta }_{\ell } (\mathsf{0}, \t -)\Asterisk _ {\bi }\AAAA _\mathtt{y}$ defines the left biadjoint to $\ell\circ J$. Moreover, $J$ is right biadjoint as well if $\aaa $ has codescent objects of these diagrams $\AAAA _\mathtt{y}$. Next section, we give precisely the diagrams $\AAAA _\mathtt{y}$ and prove a strict version of our theorem as a consequence of Lemma \ref{lemmainstrict}.

$$\xymatrix{\aaa\ar[rr]^-{J}\ar[drr]_-{R}&&\mathsf{Lax}\textrm{-}\TTTTT\textrm{-}\Alg \ar[d]\ar[r]^{\ell}&\mathsf{Lax}\textrm{-}\TTTTT\textrm{-}\Alg _\ell\ar[ld] \\
&& \bbb &}$$

\section{Lifting of biadjoints to lax algebras}\label{l}

In this section, we give our results on lifting right biadjoints to the $2$-category of lax algebras of a given pseudomonad. As explained above, we already have such results by Corollary \ref{used} and Proposition \ref{laxeffectivedescent}. But, in this section, we present an explicit calculation of the diagrams $\AAAA _\mathtt{y}$ whose lax codescent objects are needed in the construction of our left biadjoint.

\begin{defi}\label{trianglediagram}
Let $(E\dashv R, \rho , \varepsilon , v,w)$ be a biadjunction and $\TTTTT = (\TTTTT,  m  , \eta , \mu,   \iota, \tau )$ a pseudomonad on $\bbb $ such that 
$$\xymatrix{  \aaa\ar[rr]^-{J}\ar[dr]_-{R}&&\mathsf{Lax}\textrm{-}\TTTTT\textrm{-}\Alg\ar[dl]^{U}\\
&\bbb  & }$$
is commutative, in which $U$ is the forgetful $2$-functor defined in Remark \ref{forgetful}.
In this setting, for each lax $\TTTTT$-algebra $\mathtt{y}= (Y, \alg _ {{}_{\mathtt{y}}}, \overline{{\underline{\mathtt{y}}}}, \overline{{\underline{\mathtt{y}}}}_0 )$, we define the $2$-functor
$\AAAA _ \mathtt{y} : \Delta _\ell ^{\op }\to \aaa  $
\small
\begin{equation*}\tag{$\AAAA_\mathtt{y}$}
\xymatrix{  E U\mathtt{y}\ar[rrrr]|-{E(\eta _{{}_{U\mathtt{y}}})} &&&&  E\TTTTT U\mathtt{y}\ar@<3ex>[llll]^-{E(\alg _ {{}_{\mathtt{y}}})}\ar@<-3ex>[llll]_-{\varepsilon _ {{}_{EU\mathtt{y}}}E( \alg_{{}_{JEU\mathtt{y}}}\TTTTT (\rho _ {{}_{U\mathtt{y} }} ))   } 
 &&&& \ar@<3ex>[llll]^-{E\TTTTT (\alg _{{}_{\mathtt{y}}})}\ar[llll]|-{E(m _ {{}_{U\mathtt{y}}} )}\ar@<-3ex>[llll]_-{ \varepsilon _ {{}_{E\TTTTT U\mathtt{y}}}E( \alg_{{}_{JE\TTTTT U\mathtt{y}}}\TTTTT (\rho _ {{}_{\TTTTT U\mathtt{y} }} )) } E\TTTTT ^2 U\mathtt{y} }
\end{equation*}
\normalsize
in which
\small
$$
\AAAA _\mathtt{y} (\sigma _{21})        : = \eeee _ {{}_{(\alg _{{}_{\mathtt{y}}})( m _{{}_{U\mathtt{y} }}  )}}^{-1} \cdot E(\overline{\underline{\mathtt{y}} })\cdot \eeee _ {{}_{(\alg _{{}_{\mathtt{y}}})(\TTTTT (\alg _{{}_{\mathtt{y}}}) ) }}\\
\qquad\qquad\AAAA _\mathtt{y} (n_0) : = \eeee _ {{}_{(\alg _{{}_{\mathtt{y}}})(\eta _ {{}_{U\mathtt{y}}}) }}^{-1}\cdot E(\overline{{\underline{\mathtt{y}}}}_0 )\cdot \eeee _ {{}_{U\mathtt{y}}}$$
\begin{eqnarray*}
\AAAA _\mathtt{y} (n_1)        &: =& 
\left(\left(\id _ {{}_{\varepsilon _ {{}_{EU\mathtt{y} }} }}\right) \ast \left(\eeee _ {{}_{(\alg _{{}_{JEU\mathtt{y}}} \TTTTT (\rho _ {{}_{U\mathtt{y}}}) )(\eta _ {{}_{U\mathtt{y}}})  }     }^{-1}  \cdot E(\id _ {{}_{ \alg _ {{}_{JEU\mathtt{y} }}   }}\ast\eta _ {{}_{\rho _ {{}_{U\mathtt{y} }}   }} ^{-1})  
\cdot E(\overline{\underline{JEU\mathtt{y}}}_0\ast\id _ {{}_{\rho _ {{}_{U\mathtt{y} }} }}        )\right) \right)\cdot v_{{}_{U\mathtt{y}}}  \\
\AAAA _\mathtt{y} (\sigma _{20})        &: =& 
\left(\varepsilon _ {{}_{E(\alg _ {{}_{\mathtt{y} }})}}^{-1}\ast\id _ {{}_{E(\alg _ {{}_{JE\TTTTT \mathtt{y} }}\TTTTT(\rho _ {{}_{\TTTTT U\mathtt{y}}} )  ) }}\right)\cdot\left(\id _ {{}_{\varepsilon _ {{}_{EU\mathtt{y} }}   }}\ast\eeee _ {{}_{(RE(\alg _ {{}_{\mathtt{y}}}))(\alg _ {{}_{JE\TTTTT U\mathtt{y} }}\TTTTT (\rho _{{}_{\TTTTT U\mathtt{y} }} ))}}^{-1}\right)\cdot\\
&&
\left(\id _ {{}_{\varepsilon _ {{}_{EU\mathtt{y} }}   }}\ast\left(
E(\left\langle\overline{JE(\alg _ {{}_{\mathtt{y}}} })\right\rangle\ast\id _ {{}_{\TTTTT (\rho_ {{}_{\TTTTT U\mathtt{y} }} ) }})\cdot
E(\id _ {{}_{\alg _ {{}_{JEU\mathtt{y} }} }}\ast (\TTTTT\rho ) ^{-1}_{{}_{\alg _ {{}_{\mathtt{y}}} }})\right)\right)\cdot
\\
&&
\left(\id _ {{}_{\varepsilon _ {{}_{EU\mathtt{y} }}   }}\ast\left(
\eeee _ {{}_{(\alg _ {{}_{JEU\mathtt{y}}}\TTTTT (\rho _ {{}_{U\mathtt{y}}}))(\TTTTT (\alg _ {{}_{\mathtt{y}}} ))}}\right)\right)    
\end{eqnarray*}
\begin{eqnarray*}
\AAAA _\mathtt{y} (\sigma _{00})        &: =& \left(\id _ {{}_{\varepsilon _ {{}_{EU\mathtt{y} }}   }}\ast
\left(\eeee _ {{}_{(\alg _ {{}_{JEU\mathtt{y}}}\TTTTT (\rho _ {{}_{U\mathtt{y}}}))(m _ {{}_{U\mathtt{y}}} )}}^{-1}\cdot 
E(\id_ {{}_{\alg _ {{}_{JEU\mathtt{y}}} }}\ast m _ {{}_{\rho _ {{}_{U\mathtt{y} }}  }}^{-1})\cdot E(\overline{\underline{JEU\mathtt{y}}}\ast
\id _ {{}_{\TTTTT ^2 (\rho _ {{}_{U\mathtt{y}}} )}})\right)\right)\cdot\\
&&
\left(\id _ {{}_{\varepsilon _ {{}_{EU\mathtt{y} }}   }}\ast
E\left(\id _ {{}_{\alg _ {{} _{JE\TTTTT U\mathtt{y}}} }}\ast \left(\TTTTT (w_ {{}_{EU\mathtt{y}}})\cdot\tttt ^{-1} _{{}_{(R(\varepsilon _ {{}_{EU\mathtt{y}}}))(\rho 
_ {{}_{REU\mathtt{y}}}) }}\right)\ast\id _ {{}_{\TTTTT (\alg _ {{}_{JEU\mathtt{y}}})\TTTTT ^2(\rho _ {{}_{U\mathtt{y}}}) }}\right)\right)\cdot\\
&&
\left(\id _ {{}_{\varepsilon _ {{}_{EU\mathtt{y} }}   }}\ast E\left(\left\langle\overline{J(\varepsilon _ {{}_{EU\mathtt{y}}} )}\right\rangle ^{-1}\ast\id _ {{}_{\TTTTT (\rho _ {{}_{REU\mathtt{y}}} )   }}\ast\tttt _ {{}_{(\alg _ {{}_{JEU\mathtt{y}}}  ) (\TTTTT(\rho  _ {{}_{U\mathtt{y}}} ))}}^{-1}\right)\right)\cdot\\
&&
\left(\id _ {{}_{\varepsilon _ {{}_{EU\mathtt{y} }}   }}\ast  \eeee _ {{}_{(R(\varepsilon _ {{}_{EU\mathtt{y}}}))(\alg _ {{}_{JEU\TTTTT U\mathtt{y}}}\TTTTT (\rho _ {{}_{REU\mathtt{y}}})\TTTTT(\alg _{{}_{JEU\mathtt{y}}}\TTTTT (\rho  _ {{}_{U\mathtt{y}}})))      }}\right)\cdot\\
&&
\left(\varepsilon _ {{}_{\varepsilon _ {{}_{EU\mathtt{y}}} }}\ast E\left(\left(\id _ {{}_{\alg _ {{}_{JEU\TTTTT U\mathtt{y} }} }}\ast (\TTTTT \rho )_ {{}_{\alg _ {{}_{JEU\mathtt{y} }}\TTTTT (\rho _ {{}_{U\mathtt{y} }} )   }}  \right)\cdot \left(\left\langle\overline{JE(\alg _ {{}_{JEU\mathtt{y} }}\TTTTT (\rho _ {{}_{U\mathtt{y}}} ) )}\right\rangle ^{-1} \right) \right)\right)\cdot\\
&&
\left(\id _ {{}_{\varepsilon _ {{}_{EU\mathtt{y}}}\varepsilon _ {{}_{EREU\mathtt{y}}}}}\ast \eeee _ {{}_{(RE(\alg _ {{}_{JEU\mathtt{y} }}\TTTTT (\rho _ {{}_{U\mathtt{y}}} )))(\alg _ {{}_{JE\TTTTT U\mathtt{y} }}\TTTTT (\rho _ {{}_{\TTTTT U\mathtt{y}}} ))}}\right)\cdot\\
&&
\left(\id _ {{}_{\varepsilon _ {{}_{EU\mathtt{y}}}}}\ast \varepsilon _ {{}_{E(\alg _ {{}_{JEU\mathtt{y} }} \TTTTT (\rho _ {{}_{U\mathtt{y}}})) }}\ast \id _ {{}_{E(\alg _ {{}_{JE\TTTTT U\mathtt{y} }} \TTTTT (\rho _ {{}_{\TTTTT U\mathtt{y}}}))}}\right) 
\end{eqnarray*} 
\normalsize
\end{defi}

\begin{theo}[Biadjoint Triangle Theorem]\label{GOAL}
Let $(E\dashv R, \rho , \varepsilon , v,w)$ be a biadjunction, $\TTTTT = (\TTTTT,  m  , \eta , \mu,   \iota, \tau )$ a pseudomonad on $\bbb $ and $\ell :\mathsf{Lax}\textrm{-}\TTTTT\textrm{-}\Alg\to\mathsf{Lax}\textrm{-}\TTTTT\textrm{-}\Alg _ \ell $ the inclusion. Assume that
$$\xymatrix{  \aaa\ar[rr]^-{J}\ar[dr]_-{R}&&\mathsf{Lax}\textrm{-}\TTTTT\textrm{-}\Alg\ar[dl]^{U}\\
&\bbb  & }$$
is commutative. The pseudofunctor $\ell\circ J$ is right biadjoint if and only if $\aaa $ has the lax codescent object of the diagram $\AAAA _ \mathtt{y}: \Delta ^\op _ \ell\to \aaa $ for every lax $\TTTTT$-algebra $\mathtt{y}$. In this case, the left biadjoint $G$ is defined by $G\mathtt{y} = \dot{\Delta } _ \ell (\mathsf{0}, \t - ) \Asterisk _ {\bi } \AAAA _\mathtt{y} $

Furthermore, $J$ is right biadjoint if and only if $\aaa $ has the codescent object of the diagram $\AAAA _ \mathtt{y}: \Delta ^\op _ \ell\to \aaa $ for every lax $\TTTTT$-algebra $\mathtt{y}$. In this case, the left biadjoint $G '$ is defined by $G '\mathtt{y} = \dot{\Delta } (\mathsf{0}, \j - ) \Asterisk _ {\bi } \AAAA _\mathtt{y} $
\end{theo}
\begin{proof}
By Lemma \ref{lemmain}, Proposition \ref{laxeffectivedescent} and Remark \ref{indeedpseudofunctor}, it is enough to observe that, for each lax $\TTTTT$-algebra $\mathtt{y}$, there is a pseudonatural equivalence
$$\psi ^{\mathtt{y} }: \mathbb{T}^{\mathtt{y}}(-, J-)\longrightarrow \aaa(\AAAA _ \mathtt{y}-,-) $$ 
defined by
\begin{eqnarray*}
\psi ^{\mathtt{y} } _ {{}_{( \mathsf{1} ,A)}} &:= & \chi _ {{}_{(U\mathtt{y}, A)}}: \bbb (U\mathtt{y} , RA)\to \aaa (EU\mathtt{y}, A)\\
\psi ^{\mathtt{y} } _ {{}_{( \mathsf{2} ,A)}} &:= & \chi _ {{}_{(\TTTTT U\mathtt{y}, A)}}:  \bbb (\TTTTT U\mathtt{y} , RA)\to \aaa (E\TTTTT U\mathtt{y}, A)\\
\psi ^{\mathtt{y} } _ {{}_{( \mathsf{3} ,A)}} &:= & \chi _ {{}_{(\TTTTT ^2 U\mathtt{y},A)}}:  \bbb (\TTTTT ^2 U\mathtt{y} , RA)\to \aaa (E\TTTTT ^2 U\mathtt{y}, A)
\end{eqnarray*}
in which $\chi : \bbb (-,R-)\simeq \aaa (E -, -) $ is the pseudonatural equivalence corresponding to the biadjunction $(E\dashv R, \rho , \varepsilon , v,w)$ (see Remark \ref{addd}). Also,
\begin{equation*}
\begin{aligned}
(\psi ^{\mathtt{y} } _ {{}_{s^0}})_ {{}_{f}}&:=& \id _ {{}_{\varepsilon _ {{}_{A}} }}\ast\eeee _ {{}_{(f)(\eta _ {{}_{U\mathtt{y} }} )}}  \\
(\psi ^{\mathtt{y} } _ {{}_{d^1}})_ {{}_{f}}&:=& \id _ {{}_{\varepsilon _ {{}_{A}} }}\ast\eeee _ {{}_{(f)(\alg _ {{}_{\mathtt{y} }} )}}
\end{aligned}
\qquad\qquad
\begin{aligned}
(\psi ^{\mathtt{y} } _ {{}_{\partial ^1}})_ {{}_{f}}&:=& \id _ {{}_{\varepsilon _ {{}_{A}} }}\ast\eeee _ {{}_{(f)(m _ {{}_{U\mathtt{y} }} )}}\\
(\psi ^{\mathtt{y} } _ {{}_{\partial ^2 }})_ {{}_{f}}&:=& \id _ {{}_{\varepsilon _ {{}_{A}} }}\ast\eeee _ {{}_{(f)(\TTTTT(\alg _ {{}_{\mathtt{y} }}) )}}
\end{aligned}
\end{equation*}
\begin{eqnarray*}
(\psi ^{\mathtt{y} } _ {{}_{d ^0 }})_ {{}_{f}}&:=&
\left( \id _ {{}_{\varepsilon _ {{}_{A}} }} \ast \left( E(\id _ {{}_{\alg _ {{}_{JA}}  }}\ast \TTTTT ( w_ {{}_{A}} )\ast \id _ {{}_{\TTTTT (f) }})\cdot
E(\id _ {{}_{\alg _ {{}_{JA}}  }}\ast \tttt _ {{}_{(R(\varepsilon _ {{}_{A}} ))(\rho _ {{}_{RA}})   }} \ast \id _ {{}_{\TTTTT (f) }}) \right)\right)\cdot\\
&&
\left( \id _ {{}_{\varepsilon _ {{}_{A}} }} \ast
\left(E(\left\langle \overline{ J(\varepsilon _ {{}_{A}} ) }\right\rangle ^{-1}\ast \id _ {{}_{\TTTTT (\rho _ {{}_{RA}})\TTTTT (f) }})\cdot \eeee _ {{}_{(R(\varepsilon _ {{}_{A}}))(\alg _ {{}_{JERA}}\TTTTT (\rho _ {{}_{RA}})\TTTTT (f)) }} \right)\right)\cdot\\
&&
\left( \id _ {{}_{\varepsilon _ {{}_{A}}ER(\varepsilon _ {{}_{A}} ) }} \ast
\left( E(\id_{{}_{\alg _ {{}_{JER A }}}}\ast (\TTTTT \rho ) _ {{}_{f}})\cdot E(\left\langle \overline{ JE(f) }\right\rangle ^{-1}\ast \id _ {{}_{\TTTTT(\rho _ {{}_{U\mathtt{y} }} ) }})\right)\right)\cdot\\
&&
\left( \varepsilon _ {{}_{\varepsilon _ {{}_{A}}}} \ast
\eeee _ {{}_{(RE(f)) (\alg _ {{}_{JEU\mathtt{y}}}\TTTTT(\rho _ {{}_{U\mathtt{y} }} ) )}} \right)\cdot
\left( \id _ {{}_{\varepsilon _ {{}_{A}} }}\ast \varepsilon _ {{}_{f}}\ast \id _ {{}_{E(\alg _ {{}_{JEU\mathtt{y} }}\TTTTT (\rho _ {{}_{U\mathtt{y}}} ))}}\right) 
\end{eqnarray*}
\begin{eqnarray*}
(\psi ^{\mathtt{y} } _ {{}_{\partial ^0 }})_ {{}_{f}}&:=&\left( \id _ {{}_{\varepsilon _ {{}_{A}} }} \ast \left( E(\id _ {{}_{\alg _ {{}_{JA}}  }}\ast \TTTTT ( w_ {{}_{A}} )\ast \id _ {{}_{\TTTTT (f) }})\cdot
E(\id _ {{}_{\alg _ {{}_{JA}}  }}\ast \tttt _ {{}_{(R(\varepsilon _ {{}_{A}} ))(\rho _ {{}_{RA}})   }} \ast \id _ {{}_{\TTTTT (f) }}) \right)\right)\cdot\\
&&
\left( \id _ {{}_{\varepsilon _ {{}_{A}} }} \ast
\left(E(\left\langle \overline{ J(\varepsilon _ {{}_{A}} ) }\right\rangle ^{-1}\ast \id _ {{}_{\TTTTT (\rho _ {{}_{RA}})\TTTTT (f) }})\cdot \eeee _ {{}_{(R(\varepsilon _ {{}_{A}}))(\alg _ {{}_{JERA}}\TTTTT (\rho _ {{}_{RA}})\TTTTT (f)) }} \right)\right)\cdot\\
&&
\left( \id _ {{}_{\varepsilon _ {{}_{A}}ER(\varepsilon _ {{}_{A}} ) }} \ast
\left( E(\id_{{}_{\alg _ {{}_{JER A }}}}\ast (\TTTTT \rho ) _ {{}_{f}})\cdot E(\left\langle \overline{ JE(f) }\right\rangle ^{-1}\ast \id _ {{}_{\TTTTT(\rho _ {{}_{\TTTTT U\mathtt{y} }} ) }})\right)\right)\cdot\\
&&
\left( \varepsilon _ {{}_{\varepsilon _ {{}_{A}}}} \ast
\eeee _ {{}_{(RE(f)) (\alg _ {{}_{JE\TTTTT U\mathtt{y}}}\TTTTT(\rho _ {{}_{\TTTTT U\mathtt{y} }} ) )}} \right)\cdot
\left( \id _ {{}_{\varepsilon _ {{}_{A}} }}\ast \varepsilon _ {{}_{f}}\ast \id _ {{}_{E(\alg _ {{}_{JE\TTTTT U\mathtt{y} }}\TTTTT (\rho _ {{}_{\TTTTT U\mathtt{y}}} ))}}\right) 
\end{eqnarray*}
This defines a pseudonatural transformation which is a a pseudonatural equivalence, since it is objectwise an equivalence.
\end{proof}

\begin{theo}[Strict Biadjoint Triangle]\label{GOAL2}
Let $(E\dashv R, \rho , \varepsilon )$ be a $2$-adjunction, $(\TTTTT,  m  , \eta  )$ a $2$-monad on $\bbb $ and $\ell :\mathsf{Lax}\textrm{-}\TTTTT\textrm{-}\Alg\to\mathsf{Lax}\textrm{-}\TTTTT\textrm{-}\Alg _ \ell $ the inclusion. Assume that
$$\xymatrix{  \aaa\ar[rr]^-{J}\ar[dr]_-{R}&&\mathsf{Lax}\textrm{-}\TTTTT\textrm{-}\Alg\ar[dl]^{U}&\aaa\ar[rr]^-{\widetilde{J}}\ar[dr]_-{J}&&\mathsf{Lax}\textrm{-}\TTTTT\textrm{-}\Alg _ {\textrm{s}}\ar[dl] \\
&\bbb  &&&\mathsf{Lax}\textrm{-}\TTTTT\textrm{-}\Alg  & }$$
are commutative triangles, in which $\mathsf{Lax}\textrm{-}\TTTTT\textrm{-}\Alg _{\textrm{s}}\to\mathsf{Lax}\textrm{-}\TTTTT\textrm{-}\Alg  $ is the locally full inclusion of the $2$-category of lax algebras and strict $\TTTTT $-morphisms into the $2$-category of lax algebras and $\TTTTT $-pseudomorphisms. The pseudofunctor $\ell\circ J$ is right biadjoint if and only if $\aaa $ has the strict lax codescent object of the diagram $\AAAA _ \mathtt{y}: \Delta ^\op _ \ell\to \aaa $ for every lax $\TTTTT$-algebra $\mathtt{y}$. In this case, the left $2$-adjoint $G$ is defined by $G\mathtt{y} = \dot{\Delta } _ \ell (\mathsf{0}, \t - ) \Asterisk  \AAAA _\mathtt{y} $

Furthermore, $J$ is right $2$-adjoint if and only if $\aaa $ has the strict codescent object of the diagram $\AAAA _ \mathtt{y}: \Delta ^\op _ \ell\to \aaa $ for every lax $\TTTTT$-algebra $\mathtt{y}$.
\end{theo}
\begin{proof}
We have, in particular, the setting of Theorem \ref{GOAL}. Therefore, we can define $\psi $ as it is done in the last proof. However, in our setting, we get a $2$-natural transformation which is an objectwise isomorphism. Therefore $\psi $ is a $2$-natural isomorphism.

By Lemma \ref{lemmainstrict}, Proposition \ref{laxeffectivedescent} and Remark \ref{indeedpseudofunctor}, this completes our proof.
\end{proof}

\section{Coherence}\label{straightforward}

As mentioned in the introduction, \textit{the $2$-monadic approach to coherence} consists of studying the inclusions  induced by a $2$-monad $\TTTTT $ of Remark \ref{coherenceinclusions} to get general coherence results~\cite{Power, SLACK2020, Power89}. 

Given a $2$-monad $(\TTTTT,  m  , \eta  )$ on a $2$-category $\bbb $, the inclusions of Remark \ref{coherenceinclusions} and the forgetful functors of Remark \ref{forgetful} give in particular the commutative diagram below, in which $\mathsf{Ps}\textrm{-}\TTTTT\textrm{-}\Alg\to\bbb $ is right biadjoint and $\TTTTT\textrm{-}\Alg _ {\textrm{s}}\to\bbb $ is right $2$-adjoint. 

$$\xymatrix{
\TTTTT\textrm{-}\Alg _ {\textrm{s}}\ar[r]\ar[rd]
&
\mathsf{Ps}\textrm{-}\TTTTT\textrm{-}\Alg\ar[r]\ar[d]
&
\mathsf{Lax}\textrm{-}\TTTTT\textrm{-}\Alg _\ell\ar[dl]
\\
&\bbb &} $$

In this section, we are mainly concerned with the triangles involving the $2$-category of lax algebras. We refer to \cite{Lucatelli1} for the remaining triangles involving the $2$-category of pseudoalgebras.
The inclusion $\TTTTT\textrm{-}\Alg _ {\textrm{s}}\to \mathsf{Lax}\textrm{-}\TTTTT\textrm{-}\Alg _ \ell$ is also studied in \cite{SLACK2020}. Therein, it is proved that it has a left $2$-adjoint whenever the $2$-category 
$\TTTTT\textrm{-}\Alg _ {\textrm{s}}$ has the lax codescent objects of some diagrams called therein \textit{lax coherence data}. This is of course the immediate consequence of Theorem \ref{GOAL2} applied to the large triangle above.

Actually, we can study other inclusions of Remark \ref{coherenceinclusions} with the techniques of this paper. For instance, by Theorem \ref{GOAL2} and Corollary \ref{rt}, the inclusion of $\TTTTT \textrm{-}\Alg $ into any $2$-category of  $\TTTTT $-algebras and lax $\TTTTT $-morphisms of Remark \ref{coherenceinclusions} has a left biadjoint provided that  $\TTTTT \textrm{-}\Alg $ has lax codescent objects. Also, the inclusion of this $2$-category into any $2$-category of $\TTTTT $-algebras and $\TTTTT $-pseudomorphisms (\textit{i.e.} vertical arrows with domain in $\TTTTT \textrm{-}\Alg $ of Remark \ref{coherenceinclusions}) has a left biadjoint provided that $\TTTTT \textrm{-}\Alg $ has codescent objects.

In the more general context of pseudomonads, we can apply Theorem \ref{GOAL} and Theorem \ref{GOAL2} to understand precisely when the inclusions $\mathsf{Ps}\textrm{-}\TTTTT\textrm{-}\Alg\to  \mathsf{Lax}\textrm{-}\TTTTT\textrm{-}\Alg _\ell$ and $\mathsf{Ps}\textrm{-}\TTTTT\textrm{-}\Alg\to  \mathsf{Lax}\textrm{-}\TTTTT\textrm{-}\Alg$ have left biadjoints. In particular, we have:

\begin{theo}
Let $\TTTTT = (\TTTTT,  m  , \eta , \mu,   \iota, \tau )$ be a pseudomonad on a $2$-category $\bbb $. If $\mathsf{Ps}\textrm{-}\TTTTT\textrm{-}\Alg$ has lax codescent objects, then the inclusion $\mathsf{Ps}\textrm{-}\TTTTT\textrm{-}\Alg\to  \mathsf{Lax}\textrm{-}\TTTTT\textrm{-}\Alg$ has a left biadjoint.
Furthermore, if $\mathsf{Ps}\textrm{-}\TTTTT\textrm{-}\Alg$ has codescent objects, $\mathsf{Ps}\textrm{-}\TTTTT\textrm{-}\Alg\to\mathsf{Lax}\textrm{-}\TTTTT\textrm{-}\Alg $ has a left biadjoint.
\end{theo}

In particular, if $\TTTTT = (\TTTTT,  m  , \eta , \mu,   \iota, \tau )$ is a pseudomonad that preserves lax codescent objects, then $\mathsf{Ps}\textrm{-}\TTTTT\textrm{-}\Alg $ has lax codescent objects and, therefore, satisfies the hypothesis of the first part of the result above. Similarly, if $\TTTTT $ preserves codescent objects, it satisfies the hypothesis of the second part.

 \end{document}